\documentclass[11pt]{article}
\usepackage{amsmath}
\usepackage{amssymb}
\usepackage{amsfonts}
\usepackage{mathrsfs}
\usepackage{cite}

\newtheorem{theorem}{Theorem}[section]

\newtheorem{definition}{Definition}[section]

\newtheorem{lemma}{Lemma}[section]

\newtheorem{proposition}{Proposition}[section]
\newtheorem{remark}{Remark}[section]

\numberwithin{equation}{section}
\newenvironment{proof}{{\noindent \bf Proof.}}{\hfill $\Box$}
\newenvironment{proof2.1}{{\noindent \bf Proof of Theorem 2.1.}}{\hfill $\Box$}
\newenvironment{proof2.2}{{\noindent \bf Proof of Theorem 2.2.}}{\hfill $\Box$}
\newenvironment{proof2.3}{{\noindent \bf Proof of Theorem 2.3.}}{\hfill $\Box$}
\newenvironment{proof2.4}{{\noindent \bf Proof of Theorem 2.4.}}{\hfill $\Box$}
\newenvironment{proof2.5}{{\noindent \bf Proof of Theorem 2.5.}}{\hfill $\Box$}
\newenvironment{proof2.6}{{\noindent \bf Proof of Theorem 2.6.}}{\hfill $\Box$}
\usepackage[pagewise]{lineno}

\allowdisplaybreaks[4]

\begin{document}
\setlength{\baselineskip}{16pt}{\setlength\arraycolsep{2pt}}

\title{Global well-posedness of large scale moist atmosphere system with only horizontal viscosity in the dynamic equation
}
\author{Shenyang Tan$^{1,2}$,\ \  Wenjun Liu$^{1}$\footnote{Email address: wjliu@nuist.edu.cn (W. Liu), tanshenyang@njust.edu.cn (S. Tan). }\ \ \medskip\\$^{1}$School of Mathematics and Statistics,
Nanjing University of Information Science \\
 and Technology, Nanjing 210044, China\\
$^{2}$Taizhou Institute of Sci. $\&$ Tech. NJUST, Taizhou 225300, China}

\date{}
\maketitle

\begin{abstract}
    In order to find a better physical model to describe the large-scale cloud-water transformation and rainfall, we consider a moist atmosphere model consisting of the primitive equations with only horizontal viscosity in the dynamic equation and a set of humidity equations describing water vapor, rain water and cloud condensates. To overcome difficulties caused by the absence of vertical viscosity in the dynamic equation, we get the local existence of $v$ in $H^{1}$ space by combining the viscous elimination method and the $z-$weak solution method and using the generalized Bihari-Lasalle inequality. And then, we get the global existence of $v$ under higher regularity assumption of initial data. In turn, the existence of quasi-strong and strong solutions to the whole system is obtained. By introducing two new unknown quantities appropriately and utilizing the monotone operator theory to overcome difficulties caused by the Heaviside function in the source terms, we get the uniqueness of solutions.
\end{abstract}


\maketitle

\section{Introduction }
The main factors that determine cloud formation and precipitation are the thermal and dynamic processes of atmospheric movement, the content of water vapor, and the microphysical factors of cloud and precipitation. In a noninertial coordinate system, under the $(x,y,p)$ coordinates, the large-scale moist atmosphere system is formed by coupling the primitive equations
\begin{align}
\label{pe1}
&\partial_{t}v+(v\cdot\nabla)v
+\omega\partial_{p}v+\nabla\Phi+f v^{\bot}+\mathcal{A}_{v}v=S_{v},\\
 \label{pe2}&\partial_{p}\Phi=-\frac{RT}{p},\\
\label{pe3}&\nabla\cdot v+\partial_{p}\omega=0,\\
&\partial_{t}T+v\cdot\nabla T+\omega\partial_{p}T-\frac{RT}{c_{p}p}\omega+\mathcal{A}_{T}T=S_{T},\label{pe4}
\end{align}
and the conservation equation of water in the air
\begin{align}
\frac{d}{dt}q+\mathcal{A}_{q}q=S_{q}.\label{pe5}
\end{align}
Here $v,\omega,T,q$ are unknown functions, where $v=(v_{1},v_{2})$ is the horizontal velocity vector, $v^{\bot}=(-v_{2},v_{1})$, $w$ is the vertical velocity under the $(x,y,p)$ coordinates, $\nabla=(\partial_{x},\partial_{y})$, $T$ is the temperature, $\Phi$ is the geopotential, $f$ is the Coriolis force parameter, $R$ is the gas constants for dry air, $c_{p}$ is the specific heat of air at constant pressure, $p$ is the pressure, $S_{T}$ represents the sum of the heat increased or reduced by solar heating and condensation or evaporation, $S_{q}$ represents the amount of water added or removed by condensation or evaporation, $S_{v}$ is a forcing term added for mathematical generality which does not exist in reality, $\mathcal{A}_{v},\mathcal{A}_{T},\mathcal{A}_{q}$ are the viscosity terms,
$\mathcal{A}_{\ast}=-\mu_{\ast}\Delta-\nu_{\ast}\frac{\partial}{\partial p}\left(\left(\frac{gp}{R\bar{T}}\right)^{2}\frac{\partial}{\partial p}\right),
\ \ast=v,T,q,$ where $\Delta$ is the horizontal Laplacian, $\bar{T}$ is a given temperature distribution.

   Equations (\ref{pe1})-(\ref{pe4}) is the well-known primitive equations which is widely used in numerical weather forecasting of large-scale atmosphere. It was first proposed systematically in mathematics by Lions, Temam and Wang \cite{Lions} in 1992. Since then, the mathematical theory about the primitive equations has been studied by many mathematicians, see \cite{CaoTiti,CaoTiti2,CaoTiti3,CaoTiti4,CaoTiti5,CaoTiti6,Li2016,CLT2020,Ju,Ju2,kukavica1,kukavica2,You,Zhou,Gao1,Gao2,Temambook} and references therein. Recently, the primitive equations with only partial dissipation or partial viscosity have been also studied extensively and deeply. Cao and Titi\cite{CaoTiti2} studied the primitive equations with only vertical dissipation in the thermal equation and full viscosity in the dynamic equation. They obtained the global existence for strong solutions. Cao, Li and Titi \cite{CaoTiti6,CLT2020} considered the primitive equations with only horizontal eddy viscosity in the momentum equation and only horizontal diffusion in the temperature equation. They got the global existence of strong solutions with near $H^{1}$ initial data under the periodic boundary conditions. Hussein, Saal and Wrona\cite{Hussein} further studied the primitive equations with only horizontal viscosity in momentum equation with physical boundary conditions.
   They got the local existence of $z-$weak solutions and the global existence of strong solutions as well as uniqueness of solutions through the way of Galerkin approximation.  For other cases with partial viscosity or non-viscosity, we refer readers to \cite{CaoTiti3,CaoTiti4,CaoTiti5,Li2022} and
\cite{Ibrahimnoviscosity,Ghoul,CaoTiti7,Saal} respectively.

As to the moist atmosphere system, inspired by the work in \cite{Lions}, Guo
and Huang \cite{Guo1,Guo2} proposed a new mathematical formulation of the large scale moist atmosphere in 2006. They obtained the global existence of weak solutions. In the early study on moist atmospheric system, only one humidity equation was included, neither water vapor saturation nor microphysical factors were considered. In fact, when the concentration of water vapor in the air reaches a certain saturation concentration $q_{vs}$, the phase transition will occur. In the last decade, Coti Zelati and Temam et al. \cite{Zelati,Zelati2,Zelati3,Bousquet,Temam-wu,Temam-Wang} studied the saturation phenomenon and used the multi-Heaviside function $H(q_{v}-q_{vs})$ to describe the phase transition of water vapor. Hittmeir, Li, Titi, et al. \cite{Hittmeir2017,Hittmeir} and Cao, Temam, et al. \cite{Cao1} respectively studied more detailed moist atmosphere models, in which not only multi-phases but also the microphysics factors were taken into account.
   Inspired by the work of \cite{Cao1} and \cite{Hittmeir}, we also considered a multi-phase moist atmospheric system in \cite{TanLiu}, and get the well-posedness of strong solutions under the assumption that $q_{vs}$ is a constant.

  In large-scale atmospheric dynamics, strong horizontal turbulent mixing results in much higher horizontal viscosity than vertical viscosity. However, the velocity field $v$ in the above models is either a given velocity  field or is controlled by the primitive equations with full viscosity.
 Little work has been done about the study of humid atmospheric system with partial diffusion or partial viscosity. Inspired by the work mentioned above, in this paper we consider the following multi-phase moist atmosphere equations with only horizontal viscosity in the velocity equation and full dissipation in other equations:
\begin{align}\label{original system}
&\partial_{t}v+(v\cdot\nabla)v+\omega\partial_{p}v
+\nabla\Phi+fv^{\bot}-\mu_{v}\Delta v=0,\\
&\partial_{p}\Phi=-\frac{RT}{p},\\
&\nabla\cdot v+\partial_{p}w=0,\\
&\partial_{t}T+v\cdot\nabla T+w\partial_{p}T-\frac{RT}{c_{p}p}w+\mathcal{A}_{T}T=f_{T}+\frac{L}{c_{p}}(W_{cd}-W_{ev})
,\\
&\partial_{t}q_{v}+v\cdot\nabla q_{v}+w\partial_{p}q_{v}+\mathcal{A}_{q_{v}} q_{v}
=W_{ev}-W_{cd},\\
&\partial_{t}q_{c}+v\cdot\nabla q_{c}+w\partial_{p}q_{c}+\mathcal{A}_{q_{c}}q_{c}
=W_{cd}-W_{ac}-W_{cr},\\
&\partial_{t}q_{r}+v\cdot\nabla q_{r}+w\partial_{p}q_{r}+\mathcal{A}_{q_{r}} q_{r}
=g\frac{\partial}{\partial p}(\rho q_{r}V_{t})+W_{ac}+W_{cr}-W_{ev},\label{original system 11}
\end{align}
where $L$ is the latent heat of vaporization, $W_{ev}$ is the rate of evaporation of rain water, $W_{cd}$ is the condensation of water vapor to cloud water, $W_{ac}$ is the auto-conversion of cloud water vapor into rain water by accumulation of microscopic droplets, $W_{cr}$ is the collection of cloud water by falling rain, $V_{t}$ is the terminal velocity of the falling rain, $f_{T}$ is a source term for temperature variation.

Due to the lack of vertical viscosity in the dynamic equation as well as the characteristics of microphysical factors in the source term for the temperature equation, the classical methods for dealing with the $H^{1}-$regularity of velocity fields $v$ in references \cite{CaoTiti6,CLT2020,Zelati} do not work. In order to overcome this difficulty, we combine the viscous elimination method and the idea of $z-$weak solution, and use the generalized Bihari-Lasalle inequality to get the local existence of $v$ in $H^{1}$ space. And then, under higher regularity assumption of initial data, we obtain the global existence of strong solutions as in \cite{CaoTiti6}. The absence of vertical viscosity in dynamic equation makes the term $\|\partial_{p}^{2}v\|_{L^{2}(0,t;L^{2})}$ unbounded. This term must be avoided in the a priori estimates. Fortunately, we can solve this problem with the helpful inequalities introduced in \cite{CaoTiti6,CLT2020}. In addition, the emergence of the Heaviside function greatly increases the difficulty of estimating the source terms. During the proof of uniqueness of solutions, we take advantage of the monotonicity of the Heaviside function and introduce two new unknowns $Q=q_{v}+q_{c}$ and $H=T+\frac{L}{c_{p}}q_{v}$ to circumvent difficulties caused by the Heaviside function.

The rest of this paper is organized as follows.
In section 2, we give the mathematical formulation of moist atmosphere system with only horizontal viscosity in the dynamic equation and the main results about the existence of quasi-strong and strong solutions.
In section 3, we introduce an approximated system. By finding uniform boundness of approximate solutions, we get the local and global existence of quasi-strong and strong solutions.
In section 4, by introducing two new unknown quantities, we get the uniqueness of solutions.

\section{Mathematical formulation and main results}
By introducing the potential temperature \cite{Zelati}
\begin{align*}
\theta=T\left(\frac{p_{0}}{p}\right)^{R/c_{p}}-\theta_{h},
\end{align*}
where $\theta_{h}$ is a reference temperature satisfying $\theta_{h},\partial\theta_{h}/\partial p\in L^{\infty}\left((0,t)\times\mathcal{M}\right)$,
the equation for $T$ becomes
\begin{align*}
\partial_{t}\theta+v\cdot\nabla\theta+w\partial_{p}\theta=Q_{\theta},
\end{align*}
where $Q_{\theta}$ is the corresponding source term for $\theta$.

In this paper, we introduce the multi-Heaviside function to describe the transient phase transitions as in \cite{Cao1,Zelati,Temam-wu}. In addition, we use expressions in \cite{Hernandez,Deng,Majda,klemp} to describe the source terms. Namely,
\begin{align}\label{equationwithoutv}
\partial_{t}\theta+v\cdot\nabla\theta+w\partial_{p}\theta+\mathcal{A}_{\theta}\theta
&\in f_{\theta}(q_{v},q_{c},q_{r},\theta)
+\frac{L}{c_{p}\Pi}w^{-}\tilde{F}\mathcal{H}(q_{v}-q_{vs}),\nonumber\\
\partial_{t}q_{v}+v\cdot\nabla q_{v}+w\partial_{p}q_{v}+\mathcal{A}_{q_{v}}q_{v}&\in f_{q_{v}}(q_{v},q_{c},q_{r},\theta)-w^{-}F\mathcal{H}(q_{v}-q_{vs}),\nonumber\\
\partial_{t}q_{c}+v\cdot\nabla q_{c}+w\partial_{p}q_{c}+\mathcal{A}_{q_{c}}q_{c}&\in f_{q_{c}}(q_{v},q_{c},q_{r},\theta)+w^{-}F\mathcal{H}(q_{v}-q_{vs}),\nonumber\\
\partial_{t}q_{r}+v\cdot\nabla q_{r}+w\partial_{p}q_{r}+\mathcal{A}_{q_{r}}q_{r}&= f_{q_{r}}(q_{v},q_{c},q_{r},\theta).
\end{align}
Here  $w^{-}=\max\{-w,0\}$, $\Pi=(p/p_{0})^{\gamma}$, and
\begin{align}
\mathcal{A}_{\theta}=-\mu_{\theta}\Delta-\nu_{\theta}\left(p_{0}/p\right)^{R/c_{p}}
\partial_{p}\left(\left(gp/R\bar{\theta}\right)^{2}\partial_{p}\left(p_{0}/p\right)^{R/c_{p}}\right),
\end{align}
where $\gamma$ is the ratio of specific heats at constant pressure and at
constant volume, $\bar{\theta}$ is the given potential temperature profile. For the convenience of mathematical processing, we take all the viscosity coefficients as constant $1$ in this paper (since $\mathcal{A}_{q_v}=\mathcal{A}_{q_c}=\mathcal{A}_{q_r}$, for simplicity, we use $\mathcal{A}_{q}$ to represent them in section 4).
 $F$ can be expressed as
\begin{equation}\label{F}
F=F(T,p)=\frac{q_{vs}\phi(T)}{p}\left( \frac{LR-c_{p}R_{v}\phi(T)}{c_{p}R_{v}\phi(T)^{2}+q_{vs}L^{2}}\right),
\end{equation}
where
\begin{eqnarray*}
\phi(T)
\begin{cases}
=T,       & T_{\ast}\leq T\leq T^{\ast}, \\
\geq T_{\ast}/2,   & T\leq T_{\ast}, \\
0, & T\geq 2T_{\ast}.
\end{cases}
\end{eqnarray*}
Here $T_{\ast}>0$ is a given temperature smaller than any temperature on earth and $T^{\ast}$ is also a given temperature larger than any temperature on earth. It was verified in \cite{Zelati} that $F$ is uniformly bounded and Lipschitz continuous with respect to $T$. $\mathcal{H}$ is the classical multi-valued Heaviside function.
And
\begin{equation}\label{f-theta}
f_{\theta}(q_{v}, q_{c}, q_{r})=-\frac{gp}{R\Pi\theta_{h}}\frac{\partial\theta_{h}}{\partial p}w-\frac{L}{c_{p}\Pi}k_{3}\tau(q_{r})(q_{vs}-q_{v})^{+}+w\frac{\partial\theta_{h}(p)}{\partial p}+f_{\theta}^{1},
\end{equation}
\begin{equation}\label{f-qv}
f_{q_{v}}(q_{v}, q_{c}, q_{r})=k_{3}\tau(q_{r})(q_{vs}-q_{v})^{+},
\end{equation}
\begin{equation}\label{f-qc}
f_{q_{c}}(q_{v}, q_{c}, q_{r})=-k_{1}(q_{c}-q_{crit})^{+}-k_{2}q_{c}\tau(q_{r}),
\end{equation}
\begin{align}\label{f-qr}
f_{q_{r}}(q_{v}, q_{c}, q_{r})=&V_{t}\partial_{p}(\frac{p}{R\bar{\theta}}q_{r})
+k_{1}(q_{c}-q_{crit})^{+}+k_{2}q_{c}\tau(q_{r})-k_{3}\tau(q_{r})(q_{vs}-q_{v})^{+},
\end{align}
where $k_{1}, k_{2}, k_{3}$ are some dimensionless rate constants, $f_{\theta}^{1}\in L^{2}(\mathcal{M})$ is a source term, $q_{crit}$ is a given constant representing the threshold of the cloud-water mixing ratio. Considering the physical meaning
of $q_{r}(0\leq q_{r}\leq 1),$ we define $\tau(q_{r})$ as
\begin{eqnarray}\label{tau}
\tau(q_{r})=
\begin{cases}
0,&q_{r}<0,\\
q_{r},&0\leq q_{r}\leq1,\\
1,&q_{r}>1.
\end{cases}
\end{eqnarray}

In summary, we mainly consider the following multi-phase moist atmosphere system in this paper:
\begin{equation}\label{e1}
\partial_{t}v-\Delta v+(v\cdot\nabla)v+w\partial_{p}v+\nabla\Phi+fv^{\bot}=0,
\end{equation}
\begin{equation}\label{e2}
\partial_{p}\Phi+\frac{RT}{p}=0,
\end{equation}
\begin{equation}\label{e3}
\nabla\cdot v+\partial_{p}w=0,
\end{equation}
\begin{equation}\label{e4}
\partial_{t}\theta+\mathcal{A}_{\theta}\theta+v\cdot\nabla\theta+w\partial_{p}\theta\in f_{\theta}(q_{v},q_{c},q_{r},\theta)+\frac{L}{c_{p}\Pi}w^{-}\tilde{F}\mathcal{H}(q_{v}-q_{vs}),
\end{equation}
\begin{equation}\label{e5}
\partial_{t}q_{v}+\mathcal{A}_{q_{v}}q_{v}+v\cdot\nabla q_{v}+w\partial_{p}q_{v}\in f_{q_{v}}(q_{v},q_{c},q_{r},\theta)-w^{-}F\mathcal{H}(q_{v}-q_{vs}),
\end{equation}
\begin{equation}\label{e6}
\partial_{t}q_{c}+\mathcal{A}_{q_{c}}q_{c}+v\cdot\nabla q_{c}+w\partial_{p}q_{c}\in f_{q_{c}}(q_{v},q_{c},q_{r},\theta)+w^{-}F\mathcal{H}(q_{v}-q_{vs}),
\end{equation}
\begin{equation}\label{e7}
\partial_{t}q_{r}+\mathcal{A}_{q_{r}}q_{r}+v\cdot\nabla q_{r}+w\partial_{p}q_{r}= f_{q_{r}}(q_{v},q_{c},q_{r},\theta).
\end{equation}

We assume $\mathcal{M}=\mathcal{M}'\times (p_{0},p_{1})$,
where $\mathcal{M}'$ is a smooth bounded domain in $\mathbb{R}^{2}$, and $p_{0}<p_{1}$ are positive constants. The boundary of $\mathcal{M}$ is composed of $\Gamma_{i}, \Gamma_{u}, \Gamma_{l}$, where
\begin{align}
\Gamma_{i}&=\{(x,y,p)\in\bar{\mathcal{M}}:p=p_{1}\},\nonumber\\
\Gamma_{u}&=\{(x,y,p)\in\bar{\mathcal{M}}:p=p_{0}\},\nonumber\\
\Gamma_{l}&=\{(x,y,p)\in\bar{\mathcal{M}}:(x,y)\in\partial\mathcal{M}',p_{0}\leq p\leq p_{1}\}.\nonumber
\end{align}
 The boundary conditions are:
 \begin{align}\label{boundary condition}
 &{\rm on}\ \Gamma_{i}:\ \partial_{p}v=0,\ \ w=0,\ \partial_{p}\theta=\theta_{\ast}-\theta,\
 \partial_{p}q_{j}=q_{j\ast}-q_{j}, j\in \{v, c, r\};\nonumber\\
 &{\rm on}\ \Gamma_{u}:\ \partial_{p}v=0,\ w=0,\ \partial_{p}w=0,\ \partial_{p}\theta=0,\ \partial_{p}q_{j}=0,\ j\in \{v, c, r\};\\
 &{\rm on}\ \Gamma_{l}:\ v=0,\ \ \partial_{\textbf{n}}v=0,\ \  \partial_{\textbf{n}}\theta=\theta_{bl}-\theta,\ \partial_{\textbf{n}}q_{j}=q_{blj}-q_{j},\ j\in \{v, c, r\}.\nonumber
 \end{align}
 where $\textbf{n}$ is the unit normal vector of $\Gamma_{l}$, 
 $\theta_{bl}, q_{blv}, q_{blc}, q_{blr}$ are given nonnegative and sufficiently smooth functions, $\theta_{\ast}$ is a typical potential temperature, $q_{v*}, q_{c*}, q_{r*}$ are given humidity distribution, which are given nonnegative and sufficiently smooth functions.\\
In addition, The Initial conditions:
\begin{align}\label{initial condition}
v(x,y,p,0)&=v_{0}(x,y,p),\nonumber\\
\theta(x,y,p,0)&=\theta_{0}(x,y,p),\\
q_{j}(x,y,p,0)&=q_{j0}(x,y,p),\ j\in\{v, c, r\}.\nonumber
\end{align}

We denote
$\|(f_{1},\cdots,f_{n})\|_{L^{2}(\mathcal{M})}=\sum_{j=1}^{n}\|f_{j}\|_{L^{2}(\mathcal{M})}^{2},
$ and $\|f\|_{w}=\left\|\left(gp/R\bar{\theta}\right)f\right\|_{L^{2}(\mathcal{M})}.
$
Obviously, $\|f\|_{w}$ is equivalent to $\|f\|_{L^{2}}$.
For simplicity, we use $H^{s}$ to represent the classical Sobolev spaces $H^{s}(\mathcal{M})$, and use $L^{p}$ to represent the classical Lebesgue space $L^{p}(\mathcal{M}),1\leq p\leq\infty$.

As usual, we introduce the space
\begin{align*}
\mathcal{V}&=\left\{v\in C^{\infty}(\mathcal{M};\mathbb{R}^{2}):\nabla\cdot\int_{p_{0}}^{p_{1}}v(x,y,p')dp'=0,v\ \text{satisfies}\ (\ref{e1})\right\},\\
\mathbb{H}&=\text{The closure of}\ \mathcal{V}\ \text{with respect to the norm of}\ (L^{2})^{2}, \\
\mathbb{V}&=\text{The closure of}\ \mathcal{V}\ \text{with respect to the norm of}\ (H^{1})^{2}.
\end{align*}
\begin{definition}
Let $v_{0}\in\mathbb{V}$, $\theta_{0},q_{v0},q_{r0},q_{c0}\in L^{2}$. If for some $h_{q}\in L^{\infty}(\mathcal{M}\times (0,t_{1}))$ satisfying the variational inequality
\begin{align}
([\tilde{q}_{v}-q_{vs}]^{+},1)-([q_{v}-q_{vs}]^{+},1)\geq \langle h_{q},\tilde{q}_{v}-q_{v}\rangle, \ \text{a.e.}\ t\in[0,t_{1}],\forall \tilde{q}_{v}\in H^{1},
\end{align}
a solution $(v,\theta,q_{v},q_{c}.q_{r})$ to equations (\ref{e1})-(\ref{e7}) with boundary condition (\ref{boundary condition}) and initial data (\ref{initial condition}) satisfies that
\begin{align*}
&v\in C(0,t_{1};\mathbb{V}), \ \
\Delta v,\nabla\partial_{p}v\in L^{2}(0,t_{1};L^{2}),\nonumber\\
&(\theta,q_{v},q_{r},q_{c})\in C(0,t_{1};(L^{2})^{4})\cap L^{2}(0,t_{1}; (H^{1})^{4}),\nonumber\\
&\partial_{t}v\in L^{2}(0,t_{1};\mathbb{H}),\ \
\partial_{t}(\theta,q_{v},q_{c},q_{r})\in L^{2}(0,t_{1}; (H^{-1})^{4}).
\end{align*}
Then we call $(v,\theta,q_{v},q_{c},q_{r})$ a quasi-strong solution to equations (\ref{e1})-(\ref{e7}) with boundary condition (\ref{boundary condition}) and initial data (\ref{initial condition}).
\end{definition}
\begin{definition}
Let $v_{0}\in\mathbb{V},\theta_{0}, q_{v0},q_{r0},q_{c0}\in H^{1}$. A quasi-strong solution $(v,\theta,q_{v},q_{c},q_{r})$ to equations (\ref{e1})-(\ref{e7}) with boundary condition (\ref{boundary condition}) and initial data (\ref{initial condition}) is called a strong solution if
\begin{align*}
&(\theta,q_{v},q_{r},q_{c})\in C(0,t_{1};(H^{1})^{4})\cap L^{2}(0,t_{1}; (H^{2})^{4}), \nonumber\\
&\partial_{t}(v,\theta,q_{v},q_{c},q_{r})\in L^{2}(0,t_{1};\mathbb{H}\times (L^{2})^{4}).
\end{align*}
\end{definition}

We state our main results as follows.
\begin{theorem}\label{quasistrong existence global}
 Let $v_{0}\in\mathbb{V}\cap L^{\infty}$, $\partial_{p}v_{0}\in L^{m}$ for some $m>2$, $\theta_{0}, q_{v0},q_{r0},q_{c0}\in L^{2}$, and $\|v_{0}\|_{H^{1}}+\|\partial_{p}v_{0}\|_{L^{m}}+
\|v_{0}\|_{L^{\infty}}<\infty$. Then there exist at least one global in time quasi-strong solution to equations (\ref{e1})-(\ref{e7}) associated with the initial data (\ref{initial condition}) and boundary condition (\ref{boundary condition}). When $\theta_{0}, q_{v0},q_{r0},q_{c0}\geq 0$ in $\mathcal{M}$ and $F$ is replaced by its positive part $F^{+}$, the solution is unique.
\end{theorem}
\begin{theorem}\label{strong existence global}
Let $v_{0}\in \mathbb{V}\cap L^{\infty}$, $\partial_{p}v_{0}\in L^{m}$ for some $m>2$, $\theta_{0},q_{v},q_{c},q_{r}\in H^{1}$, and
 $\|v_{0}\|_{H^{1}}+\|\partial_{p}v_{0}\|_{L^{m}}+
\|v_{0}\|_{L^{\infty}}<\infty$. Then equations (\ref{e1})-(\ref{e7}) associated with the initial data (\ref{initial condition}) and boundary condition (\ref{boundary condition}) has at least one global in time strong solution $(v,\theta,q_{v},q_{c},q_{r})$. When $\theta_{0}, q_{v0},q_{r0},q_{c0}\geq 0$ in $\mathcal{M}$ and $F$ is replaced by its positive part $F^{+}$, the solution is unique.
  \end{theorem}

   Proofs of theorems are naturally divided into two parts: existence and uniqueness, which will be given in Section 3 and Section 4 respectively.

\section{Existence of quasi-strong and strong solutions}
In this section, we mainly consider the existence of solutions. We first propose a regularized approximated system that approximates equations (\ref{e1})-(\ref{e7}) in a suitable sense. By finding uniform a priori bounds for solutions and taking limit of approximated solutions, we get the existence of quasi-strong and strong solutions.
\subsection{An approximated problem }
In order to deal with the differential inclusion, as in \cite{Zelati,Temam-wu}, we select a single-valued Heaviside function $h_{q_{v}}\in \mathcal{H}(q_{v}-q_{vs})$ satisfying that
\begin{align}\label{variational inequality}
([\tilde{q}_{v}-q_{vs}]^{+},1)-([q_{v}-q_{vs}]^{+},1)\geq \langle h_{q_{v}},\tilde{q}_{v}-q_{v}\rangle, \ \text{a.e.}\ t\in[0,t_{1}],\forall \tilde{q}_{v}\in H^{1}.
\end{align}
As in \cite{Cao1,Zelati,Zelati3}, we define
\begin{eqnarray}\label{H-epsilon2}
\mathcal{H}_{\epsilon_{2}}(r)=
\begin{cases}
0,       & r\leq 0, \\
r/\epsilon_{2},   & r\in (0,\epsilon_{2}], \\
1, & r>\epsilon_{2}.
\end{cases}
\end{eqnarray}
Thus $\mathcal{H}_{\epsilon_{2}}$ can overcome the discontinuity caused by the Heaviside function $h_{q_{v}}$.
Additionally,
\begin{align}\label{H-epsilon21}
|\mathcal{H}_{\epsilon_{2}}(r)|\leq 1,\qquad |\mathcal{H}_{\epsilon_{2}}(r_{1})-\mathcal{H}_{\epsilon_{2}}(r_{2})|\leq\frac{1}{\epsilon_{2}}|r_{1}-r_{2}|,
\quad\forall r_{1}, r_{2}\in \mathbb{R}.
\end{align}

In this section we consider the following approximated problem:
\begin{equation}\label{ae1}
\partial_{t}v^{\epsilon}-\Delta v^{\epsilon}-\epsilon_{1}\partial_{p}\left(\left(\frac{gp}{R\bar{\theta}}\right)^{2}\partial_
{p}v^{\epsilon}\right)+(v^{\epsilon}\cdot\nabla)v^{\epsilon}+w^{\epsilon}\partial_{p}v^{\epsilon}+\nabla\Phi^{\epsilon}
+f{v^{\epsilon}}^{\bot}=0,
\end{equation}
\begin{equation}\label{ae2}
\partial_{t}\theta^{\epsilon}+\mathcal{A}_{\theta}\theta^{\epsilon}+v^{\epsilon}\cdot\nabla\theta^{\epsilon}
+w^{\epsilon}\partial_{p}\theta^{\epsilon}= f_{\theta^{\epsilon}}+\frac{L}{c_{p}\Pi}{w^{\epsilon}}^{-}\tilde{F}\mathcal{H}_{\epsilon_{2}},
\end{equation}
\begin{equation}\label{ae3}
\partial_{t}q_{v}^{\epsilon}+\mathcal{A}_{q_{v}}q_{v}^{\epsilon}+v^{\epsilon}\cdot\nabla q_{v}^{\epsilon}+w^{\epsilon}\partial_{p}q_{v}^{\epsilon}
= f_{q_{v}^{\epsilon}}-{w^{\epsilon}}^{-}F\mathcal{H}_{\epsilon_{2}},
\end{equation}
\begin{equation}\label{ae4}
\partial_{t}q_{c}^{\epsilon}+\mathcal{A}_{q_{c}}q_{c}^{\epsilon}+v^{\epsilon}\cdot\nabla q_{c}^{\epsilon}+w^{\epsilon}\partial_{p}q_{c}^{\epsilon}= f_{q_{c}^{\epsilon}}+{w^{\epsilon}}^{-}F\mathcal{H}_{\epsilon_{2}},
\end{equation}
\begin{equation}\label{ae5}
\partial_{t}q_{r}^{\epsilon}+\mathcal{A}_{q_{r}}q_{r}^{\epsilon}+v^{\epsilon}\cdot\nabla q_{r}^{\epsilon}+w^{\epsilon}\partial_{p}q_{r}^{\epsilon}= f_{q_{r}^{\epsilon}}.
\end{equation}
In addition, we still supplement the approximated system with boundary condition (\ref{boundary condition}) and initial condition
(\ref{initial condition}).

 Obviously, through similar argument as in \cite{TanLiu}, we can have the following well-posedness result:
 \begin{proposition}
 For any given $\epsilon=(\epsilon_{1},\epsilon_{2})$, and given $h_{q}\in L^{\infty}(\mathcal{M}\times (0,t_{1}))$ satisfying (\ref{variational inequality}),
 \begin{align*}
&(i)\ \text{If}\ v_{0}\in\mathbb{V},\theta_{0}, q_{v0},q_{r0},q_{c0}\in L^{2},\ \text{then equations (\ref{ae1})-(\ref{ae5})}\ \text{with boundary condition (\ref{boundary condition}) and}\\ \
 &\text{ initial condition
(\ref{initial condition})}\ \text{has at least one global in time quasi-strong solution }
 (v^{\epsilon},\theta^{\epsilon},q_{v}^{\epsilon},q_{c}^{\epsilon},q_{r}^{\epsilon}). \\
&(ii)\ \text{If}\  v_{0}\in\mathbb{V},\theta_{0}, q_{v0},q_{r0},q_{c0}\in H^{1},\
 \text{then there exists at least one global in time strong solution}\ \\
   &(v^{\epsilon},\theta^{\epsilon},q_{v}^{\epsilon},q_{c}^{\epsilon},q_{r}^{\epsilon})\  \text{to equations (\ref{ae1})-(\ref{ae5})}\ \text{with conditions (\ref{boundary condition}) and (\ref{initial condition})}\\ \
  \end{align*}
 \end{proposition}

In subsequent subsections, we focus on finding uniform a priori bounds for solutions respect to $\epsilon$.
\subsection{A priori $L^{2}$-estimates}
 The following two lemmas are useful in the estimation of trilinear terms.
\begin{lemma}\label{HHP}(\cite[Lemma 6.2]{Hussein})
There exists a constant $C>0$ such that
\begin{align}
\int_{\mathcal{M}}|f||g||h|d\mathcal{M}\leq C\|\nabla f\|_{L^{2}}^{\frac{1}{2}}
\| f\|_{L^{2}}^{\frac{1}{2}}\|\nabla g\|_{L^{2}}^{\frac{1}{2}}
\| g\|_{L^{2}}^{\frac{1}{2}}\left(\|h\|_{L^{2}}^{\frac{1}{2}}
\| \partial_{p}h\|_{L^{2}}^{\frac{1}{2}}+ \|h\|_{L^{2}}\right)
\end{align}
for $f,g\in L^{2}(p_{0},p_{1};H_{0}^{1}(G)),h\in H^{1}(p_{0},p_{1};L^{2}(G))$.
\end{lemma}

\begin{lemma}(\cite[Lemma 2.1]{CaoTiti6})\label{trilinear term lemma}
For any $f,g,h \in H^1$, the following inequality holds true:
\begin{align*}
 \int_{\mathcal{M}'} \int_{p_{0}}^{p_{1}} fdp\int_{p_{0}}^{p_{1}} ghdpd\mathcal{M}'
\leq& {\rm{min}}\left\{C \|f\|_{L^2}^{\frac{1}{2}}\left(\|f\|_{L^{2}}^{\frac{1}{2}} + \|\nabla f\|_{L^{2}}^{\frac{1}{2}}\right)
\|g\|_{L^{2}} \|h\|_{L^{2}}^{\frac{1}{2}}\left(\|h\|_{L^{2}}^{\frac{1}{2}}+ \|\nabla h\|_{L^{2}}^{\frac{1}{2}}\right),\right.\\
&\left.C\|f\|_{L^2}\|g\|_{L^{2}}^{\frac{1}{2}}
\left(\|g\|_{L^{2}}^{\frac{1}{2}} + \|\nabla g\|_{L^{2}}^{\frac{1}{2}}\right) \|h\|_{L^{2}}^{\frac{1}{2}}\left(\|h\|_{L^{2}}^{\frac{1}{2}}+ \|\nabla h\|_{L^{2}}^{\frac{1}{2}}\right)\right\}.
\end{align*}
\end{lemma}

Next we introduce the generalized Bihari-Lasalle inequality, which is useful during the proof of local existence of solutions.
\begin{lemma}(\cite[Theorem 28]{Dragomir})\label{nonlinearGronwall}
Let $\beta(t):[t_{0},\infty)\rightarrow [0,\infty)$ be continuous function,
$f(t)$ be nonnegative differentiable function, $\alpha(t)$ be nonnegative and nonincreasing differentiable functions, $g(s):[0,\infty)\rightarrow (0,\infty)$ be strictly increasing function, $u(t):[t_{0},\infty)\rightarrow [0,\infty)$ be continuous function. Suppose that
\begin{align}
u(t)\leq f(t)+\alpha(t)\int_{t_{0}}^{t}\beta(s)g(u(s))ds
\end{align}
and $f'(t)\left[\frac{1}{g(\gamma(t))}-1\right]\leq 0 $ on $[t_{0},\infty)$ for every nonnegative continuous function $\gamma$. Then
\begin{align}\label{nonlinearGronwall1}
u(t)\leq \mathcal{G}^{-1}\{\mathcal{G}(f(t_{0}))
+\int_{t_{0}}^{t}[\alpha(s)\beta(s)+f'(s)]ds\}
\end{align}
where
\begin{align}
\mathcal{G}(r)=\int_{r_{0}}^{r}\frac{1}{g(s)}ds,r\geq r_{0}\geq0,
\end{align}
and (\ref{nonlinearGronwall1}) holds for all values of $t$ for which the function
\begin{align}
r(t)=\mathcal{G}[f(t_{0})]+\int_{t_{0}}^{t}[\alpha(s)\beta(s)+f'(s)]ds
\end{align}
belongs to the domain of the inverse function $\mathcal{G}^{-1}$.
\end{lemma}

Fix $\epsilon>0$, by virtue of the similar $L^{2}$ a priori estimates carried out in \cite{Cao1,Zelati}, we can immediately conclude the following result:
\begin{lemma}\label{l2prior}
Let $\epsilon>0$, $t_{1}>0$, $\theta_{0},q_{v0},q_{c0},q_{r0}\in L^{2},v_{0}\in\mathbb{H}$. Then there exists at least one solution $(v^{\epsilon}, \theta^{\epsilon}, q_{v}^{\epsilon}, q_{c}^{\epsilon}, q_{r}^{\epsilon})$ to equations
(\ref{ae1})-(\ref{ae5}) with initial data (\ref{initial condition}) and boundary condition (\ref{boundary condition}). Moreover
\begin{align}\label{L2-v-theta-qj}
 \sup_{t\in [0,t_{1}]}\|(v^{\epsilon},\theta^{\epsilon},q_{j}^{\epsilon})\|_{L^{2}}^{2}
 +\int_{0}^{t_{1}}\|(\nabla v^{\epsilon},\nabla\theta^{\epsilon},\nabla q_{j}^{\epsilon})\|_{L^{2}}^{2}+\|(\epsilon\partial_{p}v^{\epsilon},\partial_{p}\theta^{\epsilon}
 ,\partial_{p}q_{j}^{\epsilon})\|_{w}^{2}dt\leq C_{1},
 \end{align}
 for any $t_{1}\geq0$, $j\in \{v,c,r\}$, where $C_{1}$ is independent of $\epsilon$
 and depends on the initial value.
\end{lemma}
\begin{proof}
 By taking the inner product of equation (\ref{ae1}) with $v$ in $L^{2}(\mathcal{M})$ and using the integration by parts, we have
 \begin{align*}
 \frac{1}{2}\frac{d}{dt}\|v^{\epsilon}\|_{L^{2}}^{2}+\|\nabla v^{\epsilon}\|_{L^{2}}^{2}
+\epsilon_{1}\|\partial_{p}v^{\epsilon}\|_{w}^{2}
=-\int_{\mathcal{M}}\nabla\Phi^{\epsilon}\cdot v^{\epsilon}d\mathcal{M},
 \end{align*}
 where we have used
 \begin{align*}
\int_{\mathcal{M}}\left[(v^{\epsilon}\cdot\nabla)v^{\epsilon}+\omega^{\epsilon}\partial_{p}v^{\epsilon}\right]\cdot v^{\epsilon} d\mathcal{M}=0,
\end{align*}
and
\begin{align*}
\int_{\mathcal{M}}f\textbf{\emph{k}}{v^{\epsilon}}^{\bot}\cdot v^{\epsilon}d\mathcal{M}=0.
\end{align*}
Considering the relation between $\Phi$ and $T$ in (\ref{e2}), through a similar argument as in \cite{CaoTiti}, we have
\begin{align*}
-\int_{\mathcal{M}}\nabla\Phi^{\epsilon}\cdot v^{\epsilon}d\mathcal{M}
\leq C\|\theta^{\epsilon}\|_{L^{2}}^{2}+\frac{1}{5}\|\nabla v^{\epsilon}\|_{L^{2}}^{2}.
\end{align*}
Then we can deduce that
\begin{align}\label{L2-v}
 \frac{1}{2}\frac{d}{dt}\|v^{\epsilon}\|_{L^{2}}^{2}+\frac{4}{5}\|\nabla v^{\epsilon}\|_{L^{2}}^{2}
+\epsilon_{1}\|\partial_{p}v\|_{w}^{2}
\leq C\|v^{\epsilon}\|_{L^{2}}^{2}+C\|\theta^{\epsilon}\|_{L^{2}}^{2}.
 \end{align}

 By taking the inner product of equation (\ref{ae2}) with $\theta^{\epsilon}$, we get that
 \begin{align*}
 \frac{1}{2}\frac{d}{dt}\|\theta^{\epsilon}\|_{L^{2}}^{2}
 +\int_{\mathcal{M}}\theta^{\epsilon}\mathcal{A}_{\theta}\theta^{\epsilon}d\mathcal{M}
=\int_{\mathcal{M}}f_{\theta^{\epsilon}} \theta^{\epsilon}d\mathcal{M}+\int_{\mathcal{M}}\frac{L}{c_{p}\Pi}{w^{\epsilon}}^{-}\tilde{F}\mathcal{H}_{\epsilon_{2}} \theta^{\epsilon}d\mathcal{M}.
 \end{align*}
 Through a direct calculation and using integration by parts, we have
 \begin{align}\label{L2-theta-0}
 \int_{\mathcal{M}}\theta^{\epsilon}\mathcal{A}_{\theta}\theta^{\epsilon}d\mathcal{M}
 \geq& \|\nabla\theta^{\epsilon}\|_{L^{2}}^{2}+\|\partial_{p}\theta^{\epsilon}\|_{w}^{2}
 +\int_{\mathcal{M}}\left(\frac{gp}{R\bar{\theta}}\right)^{2}\partial_{p}\left(\frac{p_{0}}{p}\right)^{R/c_{p}}|\theta^{\epsilon}|^{2}d\mathcal{M}
 +\nonumber\\
  &2\int_{\mathcal{M}}\left(\frac{gp}{R\bar{\theta}}\right)^{2}\left(\frac{p_{0}}{p}\right)^{R/c_{p}}\theta^{\epsilon}\partial_{p}\theta^{\epsilon}d\mathcal{M}
-\int_{\Gamma_{l}}\theta_{bl}\theta^{\epsilon}d\Gamma_{l}-
\int_{\Gamma_{i}}\theta_{\ast}\theta^{\epsilon}d\Gamma_{i}.
 \end{align}
 It is obviously that
 \begin{align}\label{L2-theta-1}
 \int_{\mathcal{M}}\left(\frac{gp}{R\bar{\theta}}\right)^{2}\partial_{p}\left(\frac{p_{0}}{p}\right)^{R/c_{p}}|\theta^{\epsilon}|^{2}d\mathcal{M}
 \leq C\|\theta^{\epsilon}\|_{L^{2}}^{2}.
 \end{align}
 Utilizing the H\"older inequality and the Young inequality, we get
 \begin{align}\label{L2-theta-2}
 2\int_{\mathcal{M}}\left(\frac{gp}{R\bar{\theta}}\right)^{2}\left(\frac{p_{0}}{p}\right)^{R/c_{p}}
 \theta^{\epsilon}\partial_{p}\theta^{\epsilon}d\mathcal{M}\leq C\| \theta^{\epsilon}\|_{L^{2}}^{2}+\frac{1}{4}\|\partial_{p}\theta^{\epsilon}\|_{L^{2}}^{2},
 \end{align}
 and
 \begin{align}\label{L2-theta-3}
 \int_{\Gamma_{l}}\theta_{bl}\theta^{\epsilon}d\Gamma_{l}+
\int_{\Gamma_{i}}\theta_{\ast}\theta^{\epsilon}d\Gamma_{i}
\leq& C\|\theta\|_{L^{2}(\partial\mathcal{M})}\leq
C(\|\theta^{\epsilon}\|_{L^{2}}+\|\nabla\theta^{\epsilon}\|_{L^{2}}+\|\partial_{p}\theta^{\epsilon}\|_{L^{2}})
\nonumber\\
\leq& C(1+\|\theta^{\epsilon}\|_{L^{2}}^{2})+\frac{1}{2}\|\nabla\theta^{\epsilon}\|_{L^{2}}^{2}
+\frac{1}{4}\|\partial_{p}\theta^{\epsilon}\|_{L^{2}}^{2},
 \end{align}
 where we have used the trace inequality in the second step of (\ref{L2-theta-3}).
 Then substituting (\ref{L2-theta-1})-(\ref{L2-theta-3}) into (\ref{L2-theta-0}), we can get
 \begin{align}
 \int_{\mathcal{M}}\theta^{\epsilon}\mathcal{A}_{\theta}\theta^{\epsilon}d\mathcal{M}
 \geq\frac{1}{2}\|\nabla\theta^{\epsilon}\|_{L^{2}}^{2}+\frac{1}{2}\|\partial_{p}\theta^{\epsilon}\|_{w}^{2}
 -C(\|\theta^{\epsilon}\|_{L^{2}}^{2}+1).
 \end{align}
 Considering the definition of $f_{\theta}$, we have
 \begin{align*}
 \int_{\mathcal{M}}f_{\theta^{\epsilon}}\theta^{\epsilon}d\mathcal{M}
 \leq&\left\|-\frac{gp}{R\Pi\theta_{h}}\frac{\partial\theta_{h}}{\partial p}w^{\epsilon}-\frac{L}{c_{p}\Pi}k_{3}\tau(q_{r}^{\epsilon})(q_{vs}-q_{v}^{\epsilon})^{+}+w^{\epsilon}\frac{\partial\theta_{h}(p)}{\partial p}+f_{\theta}^{1}\right\|_{L^{2}}\|\theta^{\epsilon}\|_{L^{2}}\nonumber\\
 \leq& C(\|w^{\epsilon}\|_{L^{2}}+\|q_{v}^{\epsilon}\|_{L^{2}}+1)\|\theta^{\epsilon}\|_{L^{2}}
 \leq C(1+\|q_{v}^{\epsilon}\|_{L^{2}}^{2}+\|\theta^{\epsilon}\|_{L^{2}}^{2})+\frac{1}{20}\|\nabla v^{\epsilon}\|_{L^{2}}^{2},
 \end{align*}
 where we have used the uniform boundness of $\tau(q_{r})$.
 Considering the definition of $\tilde{F}$ and the uniform boundness of $F$ and $\mathcal{H}_{\epsilon_{2}}$, we get
 \begin{align*}
 \int_{\mathcal{M}}\frac{L}{c_{p}\Pi}{w^{\epsilon}}^{-}\tilde{F}\mathcal{H}_{\epsilon_{2}} \theta^{\epsilon}d\mathcal{M}\leq C\|w^{\epsilon}\|_{L^{2}}\|\theta^{\epsilon}\|_{L^{2}}
 \leq C\|\theta^{\epsilon}\|_{L^{2}}^{2}+\frac{1}{20}\|\nabla v^{\epsilon}\|_{L^{2}}^{2}.
  \end{align*}
  Thus
  \begin{align}\label{L2-theta}
 \frac{d}{dt}\|\theta^{\epsilon}\|_{L^{2}}^{2}
 +\|\nabla\theta^{\epsilon}\|_{L^{2}}^{2}+\|\partial_{p}\theta^{\epsilon}\|_{w}^{2}
\leq C(1+\|q_{v}^{\epsilon}\|_{L^{2}}^{2}+\|\theta^{\epsilon}\|_{L^{2}}^{2})+\frac{1}{10}\|\nabla v^{\epsilon}\|_{L^{2}}^{2}.
 \end{align}

  By similar calculations as for $\theta^{\epsilon}$, we can deal with the $L^{2}$ a priori estimates for $q_{j}^{\epsilon},j\in\{v,c,r\}$, and obtain that
  \begin{align}\label{L2-qv}
 \frac{d}{dt}\|q_{v}^{\epsilon}\|_{L^{2}}^{2}
 +\|\nabla q_{v}^{\epsilon}\|_{L^{2}}^{2}+\|\partial_{p}q_{v}^{\epsilon}\|_{w}^{2}
\leq C(\|q_{v}^{\epsilon}\|_{L^{2}}^{2}+1)+\frac{1}{10}\|\nabla v^{\epsilon}\|_{L^{2}}^{2},
 \end{align}
   \begin{align}\label{L2-qc}
 \frac{d}{dt}\|q_{c}^{\epsilon}\|_{L^{2}}^{2}
 +\|\nabla q_{c}^{\epsilon}\|_{L^{2}}^{2}+\|\partial_{p}q_{c}^{\epsilon}\|_{w}^{2}
\leq C(\|q_{c}^{\epsilon}\|_{L^{2}}^{2}+1)+\frac{1}{10}\|\nabla v^{\epsilon}\|_{L^{2}}^{2},
 \end{align}
 and
  \begin{align}\label{L2-qr}
 \frac{d}{dt}\|q_{r}^{\epsilon}\|_{L^{2}}^{2}
 +\|\nabla q_{r}^{\epsilon}\|_{L^{2}}^{2}+\|\partial_{p}q_{r}^{\epsilon}\|_{w}^{2}
\leq C(\|q_{v}^{\epsilon}\|_{L^{2}}^{2}+\|q_{c}^{\epsilon}\|_{L^{2}}^{2}+\|q_{r}^{\epsilon}\|_{L^{2}}^{2}+1).
 \end{align}

 Combining (\ref{L2-v}), (\ref{L2-theta})-(\ref{L2-qr}), we can deduce that, for $j\in\{v,c,r\}$
 \begin{align*}
 \frac{d}{dt}\|(v^{\epsilon},\theta^{\epsilon},q_{j}^{\epsilon})\|_{L^{2}}^{2}
 +\|(\nabla v^{\epsilon},\nabla\theta^{\epsilon},\nabla q_{j}^{\epsilon})\|_{L^{2}}^{2}+\|(\epsilon\partial_{p}v^{\epsilon},\partial_{p}\theta^{\epsilon},\partial_{p}q_{j}^{\epsilon})\|_{w}^{2}
\leq  C\|(v^{\epsilon},\theta^{\epsilon},q_{v}^{\epsilon},q_{c}^{\epsilon},q_{r}^{\epsilon})\|_{L^{2}}^{2}+C.
 \end{align*}
 Utilizing the Gronwall inequality, we can complete the proof.
\end{proof}
\subsection{A priori $H^{1}$-estimates for $v^{\epsilon}$}
In order to get the $H^{1}$ a priori estimate for solutions, we first consider the $H^{1}$ a priori estimate for the velocity $v^{\epsilon}$. Due to the arising of ${w^{\epsilon}}^{-}F\mathcal{H}_{\epsilon}(q_{v^{\epsilon}}-q_{vs})$ in the source term of $\theta^{\epsilon}$ and $q_{v}^{\epsilon},q_{c}^{\epsilon}$, we have to
 first study the $H^{1}$ a priori estimate for $v^{\epsilon}$ with $\theta^{\epsilon}\in L^{\infty}(0,t_{1};L^{2})\cap L^{2}(0,t_{1};H^{1})$. Based on the a priori estimate for weak solutions, we will first improve the regularity of $v^{\epsilon}$ in $p-$direction, and then the horizontal direction.
\begin{lemma}\label{vpL2}
Let $\epsilon>0$, $t_{1}>0$, $\partial_{p}v_{0}^{\epsilon}\in L^{2}(0,t_{1};\mathbb{H})$, $\theta^{\epsilon}\in L^{\infty}(0,t_{1};L^{2})\cap L^{2}(0,t_{1};H^{1})$. Then there exists  $t_{\ast}\in (0,t_{1}]$ and at least one solution $v^{\epsilon}$ to equation (\ref{ae1}) with boundary condition (\ref{boundary condition}) and initial condition (\ref{initial condition}). Moreover,
\begin{align}
 \sup_{0\leq t\leq t_{\ast}}\|\partial_{p}v^{\epsilon}\|_{L^{2}}^{2}+\int_{0}^{t_{\ast}}\|\nabla \partial_{p}v^{\epsilon}\|_{L^{2}}^{2}
+\epsilon_{1}\|\partial_{p}^{2}v^{\epsilon}\|_{w}^{2}ds
\leq C_{2},
\end{align}
where $C_{2}$ is independent of $\epsilon$.
\end{lemma}
\begin{proof}
Set $u^{\epsilon}=\partial_{p}v^{\epsilon}$. Then it satisfies that
\begin{align*}
&\partial_{t}u^{\epsilon}+(u^{\epsilon}\cdot\nabla)v^{\epsilon}+v^{\epsilon}\cdot\nabla u^{\epsilon}+(\nabla\cdot v^{\epsilon})u^{\epsilon}+w^{\epsilon}\partial_{p}u^{\epsilon}-\Delta u^{\epsilon}-\epsilon_{1}\partial_{p}^{2}\left(\left(\frac{gp}{R\bar{\theta}}\right)^{2}u^{\epsilon}\right)
+f{u^{\epsilon}}^{\bot}
=\frac{R}{p}\nabla T^{\epsilon}.
\end{align*}
Taking the inner product of the above equation with $u^{\epsilon}$ in $L^{2}$ space, we can obtain that
\begin{align}\label{vpL20}
 &\frac{1}{2}\frac{d}{dt}\|u^{\epsilon}\|_{L^{2}}^{2}+\|\nabla u^{\epsilon}\|_{L^{2}}^{2}
+\epsilon_{1}\|\partial_{p}u^{\epsilon}\|_{w}^{2}\nonumber\\
=&\int_{\mathcal{M}}\left[u^{\epsilon}\cdot\nabla v^{\epsilon}+(\nabla\cdot v^{\epsilon})u^{\epsilon}\right]\cdot u^{\epsilon}d\mathcal{M}
+\int_{\mathcal{M}}\frac{R}{p}\nabla T^{\epsilon}\cdot u^{\epsilon}d\mathcal{M},
 \end{align}
 where we have used
 \begin{align*}
 \int_{\mathcal{M}}\left[(v^{\epsilon}\cdot\nabla)u^{\epsilon}+w^{\epsilon}\partial_{p}u^{\epsilon}\right]\cdot u^{\epsilon}d\mathcal{M}=0,
 \end{align*}
 and
  \begin{align*}
 \int_{\mathcal{M}}f{u^{\epsilon}}^{\bot}\cdot u^{\epsilon}d\mathcal{M}=0.
 \end{align*}
 Using the inequality in Lemma \ref{HHP}, we can obtain that
 \begin{align*}
 &\int_{\mathcal{M}}\left[u^{\epsilon}\cdot\nabla v^{\epsilon}+(\nabla\cdot v^{\epsilon})u^{\epsilon}\right]\cdot u^{\epsilon}d\mathcal{M}
 \leq C\int_{\mathcal{M}}|u^{\epsilon}||\nabla v^{\epsilon}||u^{\epsilon}|d\mathcal{M}\nonumber\\
 \leq& C\|u^{\epsilon}\|_{L^{2}}\|\nabla u^{\epsilon}\|_{L^{2}}\left(\|\nabla v^{\epsilon}\|_{L^{2}}+\|\nabla v^{\epsilon}\|_{L^{2}}^{\frac{1}{2}}\|\nabla u^{\epsilon}\|_{L^{2}}^{\frac{1}{2}} \right)\nonumber\\
 \leq& C\|\nabla v^{\epsilon}\|_{L^{2}}^{2}\left(\|u^{\epsilon}\|_{L^{2}}^{4}
 +\|u^{\epsilon}\|_{L^{2}}^{2}\right)+\frac{1}{2}\|\nabla u^{\epsilon}\|_{L^{2}}^{2}.
 \end{align*}
 By the H\"older inequality and the Young inequality, we have
 \begin{align*}
 \int_{\mathcal{M}}\frac{R}{p}\nabla T^{\epsilon}\cdot u^{\epsilon}d\mathcal{M}
\leq C\|u^{\epsilon}\|_{L^{2}}^{2}+\|\nabla\theta^{\epsilon}\|_{L^{2}}^{2}.
 \end{align*}
  Substituting the above inequalities into (\ref{vpL20}), we can deduce that
 \begin{align}
 &\frac{d}{dt}\|u^{\epsilon}\|_{L^{2}}^{2}+\|\nabla u^{\epsilon}\|_{L^{2}}^{2}
+\epsilon_{1}\|\partial_{p}u^{\epsilon}\|_{w}^{2}\nonumber\\
\leq
&C\left(\|\nabla v^{\epsilon}\|_{L^{2}}^{2}+1\right)\left(\|u^{\epsilon}\|_{L^{2}}^{4}
 +\|u^{\epsilon}\|_{L^{2}}^{2}+1\right)+\|\nabla\theta^{\epsilon}\|_{L^{2}}^{2}.
\end{align}
Then integrating on time from $0$ to $t$, we get
 \begin{align*}
 &\|u^{\epsilon}\|_{L^{2}}^{2}+\int_{0}^{t}\|\nabla u^{\epsilon}\|_{L^{2}}^{2}
+\epsilon_{1}\|\partial_{p}u^{\epsilon}\|_{w}^{2}ds\nonumber\\
\leq
&C\int_{0}^{t}\left(\|\nabla v^{\epsilon}\|_{L^{2}}^{2}+1\right)\left(\|u^{\epsilon}\|_{L^{2}}^{4}
 +\|u^{\epsilon}\|_{L^{2}}^{2}+1\right)ds+\int_{0}^{t}\|\nabla\theta^{\epsilon}\|_{L^{2}}^{2}
 ds+\|u^{\epsilon}(0)\|_{L^{2}}^{2}.
\end{align*}
Set
\begin{align*}
&f(t)=\int_{0}^{t}\|\nabla\theta^{\epsilon}\|_{L^{2}}^{2}
 +\|u^{\epsilon}(0)\|_{L^{2}}^{2},\ \ \alpha(t)=1,\nonumber\\
 &\beta(t)=\|\nabla v^{\epsilon}\|_{L^{2}}^{2}+1,\ \ g(u)=1+u^{2}+u^{4}.
 \end{align*}
 Utilizing the generalized Bihari-Lasalle inequality (\ref{nonlinearGronwall1}), we can deduce that
  \begin{align*}
 \|u^{\epsilon}\|_{L^{2}}^{2}+\int_{0}^{t}\|\nabla u^{\epsilon}\|_{L^{2}}^{2}
+\epsilon_{1}\|\partial_{p}u^{\epsilon}\|_{w}^{2}ds
\leq
\mathcal{G}^{-1}\left\{\int_{0}^{t}\|\nabla v^{\epsilon}\|_{L^{2}}^{2}+\|\nabla\theta^{\epsilon}\|_{L^{2}}^{2}
 +1ds\right\}
\end{align*}
for $t\in[0,t_{\ast}]$, where $t_{\ast}\in(0,t_{1}]$ is sufficiently small and the function $\mathcal{G}$ is defined as
  \begin{align*}
  \mathcal{G}(r)=\int_{0}^{r}\frac{1}{1+s+s^{2}}ds
  =\frac{2}{\sqrt{3}}{\rm artan}\left(\frac{1+2r}{\sqrt{3}}\right)-\frac{\pi}{3\sqrt{3}},\ \ \forall r>0.
  \end{align*}
Then considering the $L^{2}$ a priori estimate (\ref{L2-v-theta-qj}), we have
\begin{align}\label{uL2}
 \sup_{0\leq t\leq t_{\ast}}\|u^{\epsilon}\|_{L^{2}}^{2}+\int_{0}^{t_{\ast}}\|\nabla u^{\epsilon}\|_{L^{2}}^{2}
+\epsilon_{1}\|\partial_{p}u^{\epsilon}\|_{w}^{2}ds
\leq C_{2},
\end{align}
where $C_{2}$ is independent of $\epsilon$.
\end{proof}

 Next, we consider the $H^{1}$ a priori estimate for $v^{\epsilon}$. Based on the result in the above lemma, we only need to prove the $H^{1}$ estimate in the horizontal direction.

 \begin{lemma}
 Let $\epsilon>0$, $t_{1}>0$, $\theta^{\epsilon}\in L^{\infty}(0,t_{1};L^{2})\cap L^{2}(0,t_{1};H^{1})$, $v_{0}^{\epsilon}\in\mathbb{V}$. Then there exists $0< t_{\ast}\leq t_{1}$ and a solution $v^{\epsilon}$ to equation (\ref{ae1})  with boundary condition (\ref{boundary condition}) and initial condition (\ref{initial condition}) in time interval $(0,t_{\ast})$. Moreover,
\begin{align}
 \sup_{0\leq t\leq t_{\ast}}\|v^{\epsilon}\|_{H^{1}}^{2}+\int_{0}^{t_{\ast}}\|\nabla \partial_{p}v^{\epsilon}\|_{L^{2}}^{2}+\|\Delta v^{\epsilon}\|_{L^{2}}^{2}
+\epsilon_{1}\|\partial_{p}^{2}v^{\epsilon}\|_{w}^{2}ds
\leq C_{3},
\end{align}
where $C_{3}$ is independent of $\epsilon$.
 \end{lemma}
\begin{proof}
 By taking the inner product of equation (\ref{ae1}) with $-\Delta v$ in $L^{2}(\mathcal{M})$, using integration by parts, we obtain that
\begin{align*}
&\frac{1}{2}\frac{d}{dt}\|\nabla v^{\epsilon}\|_{L^{2}}^{2}+\|\Delta v^{\epsilon}\|_{L^{2}}^{2}
+\epsilon_{1}\|\nabla\partial_{p}v^{\epsilon}\|_{w}^{2}\nonumber\\
=&\int_{\mathcal{M}}\left[(v^{\epsilon}\cdot\nabla) v^{\epsilon}+w^{\epsilon}\partial_{p}v^{\epsilon}\right]\cdot \Delta v^{\epsilon}d\mathcal{M}
+\int_{\mathcal{M}}\nabla\Phi^{\epsilon}\cdot\Delta v^{\epsilon}d\mathcal{M}
+\int_{\mathcal{M}}f{v^{\epsilon}}^{\bot}\cdot\Delta v^{\epsilon}d\mathcal{M}.
 \end{align*}
 Noting the fact that
 \begin{align}
 f(p)\leq C\int_{p_{0}}^{p_{1}}|f|dp+C\int_{p_{0}}^{p_{1}}|\partial_{p}f|dp,
 \end{align}
  and utilizing the inequality in Lemma \ref{trilinear term lemma}, we can deduce that
 \begin{align*}
 &\int_{\mathcal{M}}(v^{\epsilon}\cdot\nabla)v^{\epsilon}\cdot\Delta v^{\epsilon}d\mathcal{M}\leq\int_{\mathcal{M}'} \int_{p_{0}}^{p_{1}} |v^{\epsilon}|+|\partial_{p}v^{\epsilon}|dp\int_{p_{0}}^{p_{1}} |\nabla v^{\epsilon}||\Delta v^{\epsilon}|dp\mathcal{M}'\nonumber\\
\leq&C \left(\|v^{\epsilon}\|_{L^2}^{\frac{1}{2}}+\|u^{\epsilon}\|_{L^2}^{\frac{1}{2}}\right)
\left(\|v^{\epsilon}\|_{L^{2}}^{\frac{1}{2}} + \|\nabla v^{\epsilon}\|_{L^{2}}^{\frac{1}{2}}+\|u^{\epsilon}\|_{L^2}^{\frac{1}{2}}+\|\nabla u^{\epsilon}\|_{L^2}^{\frac{1}{2}}\right)\cdot\nonumber\\
&\|\Delta v^{\epsilon}\|_{L^{2}} \|\nabla v^{\epsilon}\|_{L^{2}}^{\frac{1}{2}}\left(\|\nabla v^{\epsilon}\|_{L^{2}}^{\frac{1}{2}}+ \|\nabla^{2} v^{\epsilon}\|_{L^{2}}^{\frac{1}{2}}\right)\nonumber\\
\leq&C\left(\|v^{\epsilon}\|_{L^{2}}^{2}+\|\nabla v^{\epsilon}\|_{L^{2}}^{2}+\|v^{\epsilon}\|_{L^{2}}^{4}
+\|v^{\epsilon}\|_{L^{2}}^{2}\|\nabla v^{\epsilon}\|_{L^{2}}^{2}+\|u^{\epsilon}\|_{L^{2}}^{2}+\|\nabla u^{\epsilon}\|_{L^{2}}^{2}\right.\nonumber\\
&\left.+\|u^{\epsilon}\|_{L^{2}}^{4}+\|u^{\epsilon}\|_{L^{2}}^{2}\|\nabla u^{\epsilon}\|_{L^{2}}^{2}\right)\|\nabla v^{\epsilon}\|_{L^{2}}^{2}+\frac{1}{8}\|\Delta v^{\epsilon}\|_{L^{2}}^{2}.
 \end{align*}
 Considering the inequality in Lemma \ref{trilinear term lemma} again, we have
\begin{align*}
 &\int_{\mathcal{M}}w^{\epsilon}\partial_{p}v^{\epsilon}\cdot\Delta v^{\epsilon}d\mathcal{M}\leq\int_{\mathcal{M}'}\int_{p_{0}}^{p_{1}}|\nabla v^{\epsilon}|dp\int_{p_{0}}^{p_{1}}|\partial_{p}v^{\epsilon}||\Delta v^{\epsilon}|dpd\mathcal{M}'\nonumber\\
 \leq& C\|\Delta v^{\epsilon}\|_{L^{2}}\|\nabla v^{\epsilon}\|_{L^{2}}^{\frac{1}{2}}\left(
 \|\nabla v^{\epsilon}\|_{L^{2}}^{\frac{1}{2}}+\|\nabla^{2} v^{\epsilon}\|_{L^{2}}^{\frac{1}{2}}\right)\|u^{\epsilon}\|_{L^{2}}^{\frac{1}{2}}\left(
 \|u^{\epsilon}\|_{L^{2}}^{\frac{1}{2}}+\|\nabla u^{\epsilon}\|_{L^{2}}^{\frac{1}{2}}\right)\nonumber\\
 \leq& C\left(\|u^{\epsilon}\|_{L^{2}}^{2}+\|u^{\epsilon}\|_{L^{2}}^{4}+\|u^{\epsilon}\|_{L^{2}}^{2}
 \|\nabla u^{\epsilon}\|_{L^{2}}^{2}+\|\nabla u^{\epsilon}\|_{L^{2}}^{2}\right)\|\nabla v^{\epsilon}\|_{L^{2}}^{2}+\frac{1}{8}\|\Delta v^{\epsilon}\|_{L^{2}}^{2}.
 \end{align*}
Using the H\"older inequality and the Young inequality, as well as the Minkowski inequality in integral form, we have
\begin{align*}
\int_{\mathcal{M}}\nabla\Phi^{\epsilon}\cdot\Delta v^{\epsilon}d\mathcal{M}
&=\int_{\mathcal{M}}\nabla\Phi^{\epsilon}_{s}\cdot\Delta v^{\epsilon}d\mathcal{M}
+\int_{\mathcal{M}}\int_{p_{0}}^{p_{1}}\frac{R}{p}\nabla T^{\epsilon}dp'\cdot\Delta v^{\epsilon}d\mathcal{M}\nonumber\\
\leq& C\|\nabla\theta^{\epsilon}\|_{L^{2}}^{2}+\frac{1}{8}\|\Delta v^{\epsilon}\|_{L^{2}}^{2},
\end{align*}
where we have used the fact that
\begin{align*}
\int_{\mathcal{M}}\nabla\Phi^{\epsilon}_{s}\cdot\Delta v^{\epsilon}d\mathcal{M}
=\int_{\mathcal{M}'}\nabla\Phi^{\epsilon}_{s}\cdot\Delta \bar{v^{\epsilon}}d\mathcal{M}'=0.
\end{align*}
Here $\bar{v^{\epsilon}}$ stands for the mean value of $v^{\epsilon}$ from $p_{0}$ to $p_{1}$, and $\Phi_{s}$ is the geopotential at $p=p_{1}$.
Similarly,
\begin{align*}
\int_{\mathcal{M}}f{v^{\epsilon}}^{\bot}\cdot\Delta v^{\epsilon}d\mathcal{M}
\leq C\|v^{\epsilon}\|_{L^{2}}^{2}
+\frac{1}{8}\|\Delta v^{\epsilon}\|_{L^{2}}^{2}.
\end{align*}
Thus
\begin{align*}
\frac{d}{dt}\|\nabla v^{\epsilon}\|_{L^{2}}^{2}+\|\Delta v^{\epsilon}\|_{L^{2}}^{2}
+\epsilon_{1}\|\nabla\partial_{p}v^{\epsilon}\|_{w}^{2}
\leq A_{1}(t)\|\nabla v^{\epsilon}\|_{L^{2}}^{2}+A_{2}(t),
 \end{align*}
 where
 \begin{align*}
 A_{1}(t)=C(&\|v^{\epsilon}\|_{L^{2}}^{2}+\|\nabla v^{\epsilon}\|_{L^{2}}^{2}+\|v^{\epsilon}\|_{L^{2}}^{4}
+\|v^{\epsilon}\|_{L^{2}}^{2}\|\nabla v^{\epsilon}\|_{L^{2}}^{2}+\|u^{\epsilon}\|_{L^{2}}^{2}+\|\nabla u^{\epsilon}\|_{L^{2}}^{2}+\nonumber\\
&\|u^{\epsilon}\|_{L^{2}}^{4}
+\|u^{\epsilon}\|_{L^{2}}^{2}\|\nabla u^{\epsilon}\|_{L^{2}}^{2}),
 \end{align*}
 and
 \begin{align*}
 A_{2}(t)=C\left(\|v^{\epsilon}\|_{L^{2}}^{2}+
 \|\nabla\theta^{\epsilon}\|_{L^{2}}^{2}\right).
 \end{align*}
 Considering regularities of $v^{\epsilon},\nabla v^{\epsilon},\nabla\theta^{\epsilon},u^{\epsilon},\nabla u^{\epsilon}$ in (\ref{L2-v-theta-qj}) and (\ref{uL2}), we know that $A_{1}$ and $A_{2}$ are $L^{2}$ integrable in $(0,t_{\ast})$. Then using the Gronwall inequality, we can obtain that
 \begin{align}
\sup_{0\leq t\leq t_{\ast}}\|\nabla v^{\epsilon}\|_{L^{2}}^{2}+\int_{0}^{t_{\ast}}\|\Delta v^{\epsilon}\|_{L^{2}}^{2}
+\epsilon_{1}\|\nabla\partial_{p}v^{\epsilon}\|_{w}^{2}ds
\leq C_{3},
 \end{align}
 where $C_{3}$ is independent of $\epsilon$.
\end{proof}

 Next, we consider the time regularity of solutions to equations (\ref{ae1})-(\ref{ae5}). We first consider the time regularity of $v$.
 \begin{lemma}\label{trv}
 Let $\epsilon>0$, $t_{1}>0$, $\theta^{\epsilon}\in L^{\infty}(0,t_{1};L^{2})\cap L^{2}(0,t_{1};H^{1})$, $v_{0}^{\epsilon}\in\mathbb{V}$ and $v^{\epsilon}$ be the solution to equation (\ref{ae1}) with boundary condition (\ref{boundary condition}) and initial condition (\ref{initial condition}). Then
 \begin{align*}
 \partial_{t}v^{\epsilon}\in L^{2}(0,t_{\ast};\mathbb{H}).
 \end{align*}
 \end{lemma}
 \begin{proof}
 Choosing a test function $\tilde{v}\in L^{2}(0,t_{\ast};\mathbb{H})$ satisfying that $\|\tilde{v}\|_{L^{2}(0,t_{\ast};\mathbb{H})}\leq 1$, we can refer from equation (\ref{e1}) that
\begin{align*}
 \int_{\mathcal{M}}\partial_{t}v^{\epsilon}\cdot\tilde{v}d\mathcal{M}
 =&\int_{\mathcal{M}}\Delta v^{\epsilon}\cdot\tilde{v}d\mathcal{M}
 -\int_{\mathcal{M}}(v^{\epsilon}\cdot\nabla)v^{\epsilon}\cdot\tilde{v}+w^{\epsilon}\partial_{p}v^{\epsilon}\cdot\tilde{v}d\mathcal{M}\nonumber\\
 &-\int_{\mathcal{M}}\nabla\Phi^{\epsilon}\cdot\tilde{v}d\mathcal{M}
-\int_{\mathcal{M}}f{v^{\epsilon}}^{\bot}\cdot\tilde{v}d\mathcal{M}.
 \end{align*}
 It is easy to verify that
 \begin{align*}
 \int_{\mathcal{M}}\Delta v^{\epsilon}\cdot\tilde{v}d\mathcal{M}\leq C\|\Delta v^{\epsilon}\|_{L^{2}}\|\tilde{v}\|_{L^{2}},
 \end{align*}
   \begin{align*}
-\int_{\mathcal{M}}\nabla\Phi^{\epsilon}\cdot\tilde{v}d\mathcal{M}\leq C\|\theta^{\epsilon}\|_{H^{1}}\|\tilde{v}\|_{L^{2}},
 \end{align*}
 and
 \begin{align*}
-\int_{\mathcal{M}}f{v^{\epsilon}}^{\bot}\cdot\tilde{v}d\mathcal{M}\leq C\|v^{\epsilon}\|_{L^{2}}\|\tilde{v}\|_{L^{2}}.
 \end{align*}

 Then we consider the trilinear term. Through similar arguments as in (\ref{qvH1h1}) and (\ref{qvH1h2}), we can deduce that
 \begin{align*}
 &\int_{\mathcal{M}}(v^{\epsilon}\cdot\nabla)v^{\epsilon}\cdot\tilde{v}d\mathcal{M}\leq\int_{\mathcal{M}'} \int_{p_{0}}^{p_{1}} |v^{\epsilon}|+|\partial_{p}v^{\epsilon}|dp\int_{p_{0}}^{p_{1}} |\nabla v^{\epsilon}||\tilde{v}|dp\mathcal{M}'\nonumber\\
\leq& C \left(\|v^{\epsilon}\|_{L^2}^{\frac{1}{2}}+\|u^{\epsilon}\|_{L^2}^{\frac{1}{2}}\right)
\left(\|v^{\epsilon}\|_{L^{2}}^{\frac{1}{2}} + \|\nabla v^{\epsilon}\|_{L^{2}}^{\frac{1}{2}}+\|u^{\epsilon}\|_{L^2}^{\frac{1}{2}}+\|\nabla u^{\epsilon}\|_{L^2}^{\frac{1}{2}}\right)\|\nabla v^{\epsilon}\|_{L^{2}}^{\frac{1}{2}}\left(\|\nabla v^{\epsilon}\|_{L^{2}}^{\frac{1}{2}}+\right.\nonumber\\
&\left.\|\nabla^{2} v^{\epsilon}\|_{L^{2}}^{\frac{1}{2}}\right)\|\tilde{v}\|_{L^{2}}
\leq C\left(\|v^{\epsilon}\|_{H^{1}}^{2}+\|v^{\epsilon}\|_{H^{1}}\|\Delta v^{\epsilon}\|_{L^{2}}\right)\|\tilde{v}\|_{L^{2}},
 \end{align*}
 and
\begin{align*}
 &\int_{\mathcal{M}}w^{\epsilon}\partial_{p}v^{\epsilon}\cdot\tilde{v}d\mathcal{M}
 \leq\int_{\mathcal{M}'}\int_{p_{0}}^{p_{1}}|\nabla v^{\epsilon}|dp\int_{p_{0}}^{p_{1}}|\partial_{p}v^{\epsilon}||\tilde{v}|dpd\mathcal{M}'\nonumber\\
 \leq& C\|\tilde{v}\|_{L^{2}}\|\nabla v^{\epsilon}\|_{L^{2}}^{\frac{1}{2}}\left(
 \|\nabla v^{\epsilon}\|_{L^{2}}^{\frac{1}{2}}+\|\nabla^{2} v^{\epsilon}\|_{L^{2}}^{\frac{1}{2}}\right)\|\partial_{p}v^{\epsilon}\|_{L^{2}}^{\frac{1}{2}}\left(
 \|\partial_{p}v^{\epsilon}\|_{L^{2}}^{\frac{1}{2}}+\|\nabla \partial_{p}v^{\epsilon}\|_{L^{2}}^{\frac{1}{2}}\right)\nonumber\\
 \leq& C\left(\|v^{\epsilon}\|_{H^{1}}^{2}+\|v^{\epsilon}\|_{H^{1}}^{3}+\|\Delta v^{\epsilon}\|_{L^{2}}+\|\nabla\partial_{p}v^{\epsilon}\|_{L^{2}}+\|v^{\epsilon}\|_{H^{1}}\|\Delta v^{\epsilon}\|_{L^{2}}+ \|v^{\epsilon}\|_{H^{1}}\|\nabla\partial_{p}v^{\epsilon}\|_{L^{2}}\right)\|\tilde{v}\|_{L^{2}}.
 \end{align*}
 Combining the above inequalities, we have
 \begin{align*}
 \int_{\mathcal{M}}\partial_{t}v^{\epsilon}\cdot\tilde{v}d\mathcal{M}\leq A_{3}(t)\|\tilde{v}\|_{L^{2}}\leq A_{3}(t),
 \end{align*}
 where
 \begin{align*}
 A_{3}(t)=&\|\Delta v^{\epsilon}\|_{L^{2}}+\|\theta^{\epsilon}\|_{H^{1}}+\|v^{\epsilon}\|_{L^{2}}+
 \|v^{\epsilon}\|_{H^{1}}^{2}+\|v^{\epsilon}\|_{H^{1}}\|\Delta v^{\epsilon}\|_{L^{2}}+\|v^{\epsilon}\|_{H^{1}}^{2}+\|v^{\epsilon}\|_{H^{1}}^{3}\nonumber\\
 &+\|\Delta v^{\epsilon}\|_{L^{2}}+\|\nabla\partial_{p}v^{\epsilon}\|_{L^{2}}+\|v^{\epsilon}\|_{H^{1}}\|\Delta v^{\epsilon}\|_{L^{2}}+ \|v^{\epsilon}\|_{H^{1}}\|\nabla\partial_{p}v^{\epsilon}\|_{L^{2}}.
 \end{align*}
 Considering regularities of $v^{\epsilon}$ and $\theta^{\epsilon}$, we can know that $A_{3}(t)$ is $L^{2}$ integrable in $(0,t_{\ast})$. Then we can deduce that
 $\partial_{t}v^{\epsilon}\in L^{2}(0,t_{\ast};\mathbb{H})$.
\end{proof}

 As for the time regularities of $\theta,q_{v},q_{c},q_{r}$, we can deal with them in a similar way. For more detailed procedures, we refer readers to \cite{Zelati} or \cite{TanLiu}. It should be noted that the treatment of trilinear terms can be referred to (\ref{qvH1h1}) and (\ref{qvH1h2}). As a conclusion, we can get that
 \begin{align}
 \partial_{t}\theta,\partial_{t}q_{j}\in L^{2}(0,t_{\ast};H^{-1}),j\in \{v,c,r\}.
 \end{align}
So far, we have obtained the local existence of quasi-strong solutions to the approximated system.
 \begin{proposition}
 Let $v_{0}\in\mathbb{V}$, $\theta_{0}, q_{v0},q_{r0},q_{c0}\in L^{2}$, and $\epsilon>0$ be fixed. Then there exists $t_{\ast}>0\ (t_{\ast}\leq t_{1}, \text{independent of $\epsilon$})$, and a quasi-strong solution $(v^{\epsilon},\theta^{\epsilon},q_{v}^{\epsilon},q_{c}^{\epsilon},q_{r}^{\epsilon})$ to equations (\ref{ae1})-(\ref{ae5}) with boundary condition (\ref{boundary condition}) and initial condition (\ref{initial condition}) satisfying that
\begin{align*}
&v^{\epsilon}\in L^{\infty}(0,t_{\ast};\mathbb{V}),\ \Delta v^{\epsilon},\nabla\partial_{p}v^{\epsilon}\in L^{2}(0,t_{\ast}; (L^{2})^{2}), \nonumber\\
&(\theta_{\epsilon},q_{v}^{\epsilon},q_{c}^{\epsilon},q_{r}^{\epsilon})\in C(0,t_{\ast};(L^{2})^{4})\cap L^{2}(0,t_{\ast}; (H^{1})^{4}).
\end{align*}
Moreover,
\begin{align}\label{quasistrong bound}
&\|v^{\epsilon}\|
_{L^{\infty}(0,t_{\ast};\mathbb{V})}
+\|\left(\theta^{\epsilon},q_{v}^{\epsilon},q_{c}^{\epsilon},q_{r}^{\epsilon}\right)\|_{L^{\infty}(0,t_{\ast};L^{2})}
+\|(\Delta v^{\epsilon},\nabla\partial_{p}v^{\epsilon},\epsilon\partial_{p}^{2}v^{\epsilon})\|_{L^{2}(0,t_{\ast};L^{2})}\nonumber\\
&+\|\left(\theta^{\epsilon},q_{v}^{\epsilon},q_{c}^{\epsilon},q_{r}^{\epsilon}\right)\|_{L^{2}(0,t_{\ast};H^{1})}
+\|\partial_{t}v^{\epsilon}\|
_{L^{2}(0,t_{\ast};\mathbb{H})}
+\|\partial_{t}(\theta^{\epsilon},q_{v}^{\epsilon},q_{c}^{\epsilon},q_{r}^{\epsilon})\|_{L^{2}\left(0,t_{\ast};H^{-1}\right)}
\leq \mathcal{C},
\end{align}
where $\mathcal{C}$ is monotonically increasing positive function with respect to $t_{\ast}$, which is independent of $\epsilon$, and depends on the initial data.
\end{proposition}
\subsection{The local existence of quasi-strong solutions}
In this subsection, we obtain the local existence of quasi-strong solutions to equations (\ref{e1})-(\ref{e7}) by passing to the limit in the approximated equations  (\ref{ae1})-(\ref{ae5}) as $\epsilon\rightarrow 0$. As usual, we denote strong, weak, and weak-$\ast$ convergence as $\epsilon\rightarrow 0$ by $\rightarrow,\rightharpoonup,\stackrel{\ast}{\rightharpoonup}$ respectively. Considering inequality (\ref{quasistrong bound}) and Aubin-Lions compactness theorem \cite{Aubin,Lions2}, we can deduce the existence of a subsequence, still denoted by $v^{\epsilon},\theta^{\epsilon},q_{v}^{\epsilon},q_{c}^{\epsilon},q_{r}^{\epsilon},$ and
\begin{align}
&v\in C(0,t_{\ast};\mathbb{V}), \ \Delta v,\nabla\partial_{p}v\in L^{2}(0,t_{\ast};L^{2}),\ \partial_{t}v\in L^{2}(0,t_{\ast};\mathbb{H}),\nonumber\\
&\theta,q_{v},q_{c},q_{r}\in C(0,t_{\ast};L^{2})\cap L^{2}(0,t_{\ast};H^{1}),\ \partial_{t}\theta,\partial_{t}q_{v},\partial_{t}q_{c},\partial_{t}q_{r}\in L^{2}(0,t_{\ast};H^{-1})
\end{align}
such that, as $\epsilon\rightarrow 0$,
\begin{align*}
&\ v^{\epsilon}\stackrel{\ast}{\rightharpoonup} v\  \text{in}\ C(0,t_{\ast};\mathbb{V}), \\
 &(\Delta v^{\epsilon},\nabla\partial_{p}v^{\epsilon}){\rightharpoonup} (\Delta v,\nabla\partial_{p}v) \ \text{in}\  L^{2}(0,t_{\ast};(L^{2})^{2}),\\ &\partial_{t}v^{\epsilon}\rightharpoonup\partial_{t} v\ \text{in}\ L^{2}(0,t_{\ast};\mathbb{H}),
\end{align*}
and
\begin{align*}
&\left(\theta^{\epsilon},q_{v}^{\epsilon},q_{c}^{\epsilon},q_{r}^{\epsilon}\right)\rightharpoonup\left(\theta,q_{v},q_{c},q_{r}\right)
\ \text{in}\ \ L^{2}(0,t_{\ast};(H^{1})^{4}),\\
&\left(\theta^{\epsilon},q_{v}^{\epsilon},q_{c}^{\epsilon},q_{r}^{\epsilon}\right)\stackrel{\ast}{\rightharpoonup}\left(\theta,q_{v},q_{c},q_{r}\right)
\ \text{in}\ \ L^{\infty}(0,t_{\ast};(L^{2})^{4}),\ \nonumber\\
&\partial_{t}(\theta^{\epsilon},q_{v}^{\epsilon},q_{c}^{\epsilon},q_{r}^{\epsilon})\rightharpoonup \partial_{t}(\theta,q_{v},q_{c},q_{r})\ \text{in}\ L^{2}\left(0,t_{\ast};(H^{-1})^{4}\right),\\
&\left(\theta^{\epsilon},q_{v}^{\epsilon},q_{c}^{\epsilon},q_{r}^{\epsilon}\right)\rightarrow\left(\theta,q_{v},q_{c},q_{r}\right)
\ \text{in}\ \ L^{2}(0,t_{\ast};(L^{2})^{4}).
\end{align*}

As to the proof of convergence between the approximated system (\ref{ae1})-(\ref{ae5}) and the original system (\ref{e1})-(\ref{e7}), we refer readers to \cite{CaoTiti6}, \cite{Cao1} and \cite{Zelati}.
In fact, we have proved the local existence of quasi-strong solutions of system (\ref{e1})-(\ref{e7}).
\begin{proposition}\label{quasistrong existence local}
 Let $v_{0}\in\mathbb{V}$, $\theta_{0}, q_{v0}, q_{r0},q_{c0}\in L^{2}(\mathcal{M})$. Then there exists $t_{\ast}>0\ (t_{\ast}\leq t_{1}$), and a quasi-strong solution $(v,\theta,q_{v},q_{c},q_{r})$ to equations (\ref{e1})-(\ref{e7}) with boundary condition (\ref{boundary condition}) and initial condition (\ref{initial condition}) in time interval $(0,t_{\ast})$. Moreover,
\begin{align*}
&v\in C(0,t_{\ast};\mathbb{V}),\ \ \Delta v,\nabla\partial_{p}v\in L^{2}(0,t_{\ast}; L^{2}), \nonumber\\
&(\theta,q_{v},q_{c},q_{r})\in C(0,t_{\ast};(L^{2})^{4})\cap L^{2}(0,t_{\ast}; (H^{1})^{4}),\nonumber\\
&\partial_{t}v\in L^{2}(0,t_{\ast};\mathbb{H}),\ \
\partial_{t}(\theta,q_{v},q_{c},q_{r})\in L^{2}(0,t_{\ast}; (H^{-1})^{4}).
\end{align*}
\end{proposition}
\subsection{The local existence of strong solutions}
Compared with the case where the velocity equation is full viscosity, the most difficulty during the proof of existence of strong solution is caused by the absence of vertical dissipation in the velocity fields, which makes that $\|\partial_{p}^{2}v\|_{L^{2}(0,t;L^{2})}$ can not be bounded. Therefore, our main work is to ensure that the a priori estimates does not include this item. Here, we take $q_{v}$ as an example to illustrate how to get the $H^{1}$ regularity. As for other quantities, one can get the $H^{1}$ regularity by combining calculations in literatures \cite{Zelati,TanLiu,Hittmeir} and this paper.
\begin{lemma}\label{H1}
Let $v_{0}\in\mathbb{V}$, $\theta_{0}, q_{v0}, q_{r0},q_{c0}\in H^{1}$. Then there exists $t_{\ast}\in (0,t_{1})$ and a solution $q_{v}$ to equation (\ref{e5}) with boundary condition (\ref{boundary condition}) and initial condition (\ref{initial condition}). Moreover
\begin{align}
\sup_{0\leq t\leq t_{\ast}}\|q_{v}\|_{H^{1}}^{2}
+\int_{0}^{t_{\ast}}\|q_{v}\|_{H^{2}}^{2}dt\leq C_{4},
\end{align}
where $C_{4}$ depends on the initial data and $t_{\ast}$.
\end{lemma}
\begin{proof}
By taking the inner product of equation (\ref{ae3}) with $-\Delta q_{v}$ in the $L^{2}$ space, we can infer that
\begin{align}\label{qvH1H}
&-\int_{\mathcal{M}}\partial_{t}q_{v}\Delta q_{v}d\mathcal{M}
-\int_{\mathcal{M}}\mathcal{A}_{q_{v}}q_{v}\Delta q_{v}d\mathcal{M}\nonumber\\
= &\int_{\mathcal{M}}\left(v\cdot\nabla q_{v}+w\partial_{p}q_{v}\right)\Delta q_{v}d\mathcal{M}
-\int_{\mathcal{M}}\left(f_{q_{v}}-w^{-}Fh_{q_{v}}\right)\Delta q_{v}d\mathcal{M}.
\end{align}
Utilizing similar arguments as in \cite{Hittmeir} and \cite{TanLiu}, we can obtain that
\begin{align*}
-\int_{\mathcal{M}}\partial_{t}q_{v}\Delta q_{v}d\mathcal{M}
\geq\frac{d}{dt}\left(\frac{1}{2}\|\nabla q_{v}\|_{L^{2}}^{2}+
\alpha_{lv}\int_{\Gamma_{l}}\left(\frac{1}{2}(q_{v})^{2}-q_{v}q_{blv}\right)d\Gamma_{l}\right)
-C(1+\|q_{v}\|_{L_{2}}^{2}+\|\nabla q_{v}\|_{L_{2}}^{2}),
\end{align*}
and
\begin{align*}
-\int_{\mathcal{M}}\mathcal{A}_{q_{v}}q_{v}\Delta q_{v}d\mathcal{M}
\geq\|\Delta q_{v}\|_{L^{2}}^{2}+\|\nabla\partial_{p}q_{v}\|_{w}^{2}-C.
\end{align*}
For the source term, recalling definitions of $F,f_{q_{v}}$ and $h_{q_{v}}$, we have
\begin{align}\label{qvH1050}
&\int_{\mathcal{M}}\left(f_{q_{v}}-{w}^{-}Fh_{q_{v}}\right)\Delta q_{v}d\mathcal{M}
\leq\|f_{q_{v}}-{w}^{-}Fh_{q_{v}}\|_{L^{2}}\|\Delta q_{v}\|_{L^{2}}
\nonumber\\
\leq&\left(\left\|k_{3}\tau(q_{r})(q_{vs}-q_{v})^{+}\right\|_{L^{2}}
+\left\|{w}^{-}Fh_{q_{v}}\right\|_{L^{2}}\right)
\|\Delta q_{v}\|_{L^{2}}.
\end{align}
Using the fact that $q_{vs}$ is a constant and $\tau(q_{r})$ is uniformly bounded, we have
\begin{align}\label{fqv1}
\left\|k_{3}\tau(q_{r})(q_{vs}-q_{v})^{+}\right\|_{L^{2}}\leq C(1+\|q_{v}\|_{L^{2}}^{2}).
\end{align}
As verified in \cite{Zelati}, the function $F$ and $h_{q}$ are uniformly bounded. Then
\begin{align}\label{F1}
\left\|w^{-}Fh_{q_{v}}\right\|_{L^{2}}
\leq C\|w\|_{L^{2}}\leq C\|\nabla v\|_{L^{2}},
\end{align}
where we have used $\partial_{p}w=-\nabla\cdot v$. Thus
\begin{align}
\int_{\mathcal{M}}\left(f_{q_{v}}-{w}^{-}Fh_{q_{v}}\right)\Delta q_{v}d\mathcal{M}
\leq C(1+\|q_{v}\|_{L^{2}}^{2}+\|\nabla v\|_{L^{2}}^{2})+\frac{1}{8}\|\Delta q_{v}\|_{L^{2}}^{2}.
\end{align}

Then we consider the trilinear term
\begin{align*}
\int_{\mathcal{M}}\left(v\cdot\nabla q_{v}+w\partial_{p}q_{v}\right)\Delta q_{v}d\mathcal{M}.
\end{align*}
Noting the fact that
\begin{align*}
f(p)\leq C\int_{p_{0}}^{p_{1}}|f|dp+C\int_{p_{0}}^{p_{1}}|\partial_{p}f|dp
 \end{align*}
 and using the inequality in Lemma \ref{trilinear term lemma}, we can deduce that
 \begin{align}\label{qvH1h1}
 &\int_{\mathcal{M}}(v\cdot\nabla)q_{v}\cdot\Delta q_{v}d\mathcal{M}\leq\int_{\mathcal{M}'} \int_{p_{0}}^{p_{1}} |v|+|\partial_{p}v|dp\int_{p_{0}}^{p_{1}} |\nabla q_{v}||\Delta q_{v}|dp\mathcal{M}'\nonumber\\
\leq& C \left(\|v\|_{L^2}^{\frac{1}{2}}+\|u\|_{L^2}^{\frac{1}{2}}\right)
\left(\|v\|_{L^{2}}^{\frac{1}{2}} + \|\nabla v\|_{L^{2}}^{\frac{1}{2}}+\|u\|_{L^2}^{\frac{1}{2}}+\|\nabla u\|_{L^2}^{\frac{1}{2}}\right)\|\Delta q_{v}\|_{L^{2}}\cdot\nonumber\\
&\ \ \  \|\nabla q_{v}\|_{L^{2}}^{\frac{1}{2}}\left(\|\nabla q_{v}\|_{L^{2}}^{\frac{1}{2}}+ \|\nabla^{2} q_{v}\|_{L^{2}}^{\frac{1}{2}}\right)\nonumber\\
\leq&C\left(\|v\|_{L^{2}}^{2}+\|u\|_{L^{2}}^{2}+\|\nabla v\|_{L^{2}}^{2}+\|\nabla u\|_{L^{2}}^{2}
+\|v\|_{L^{2}}^{4}+\|\nabla v\|_{L^{2}}^{2}\|v\|_{L^{2}}^{2}+\|u\|_{L^{2}}^{2}\|\nabla v\|_{L^{2}}^{2}\right.\nonumber\\
&\left.+\|u\|_{L^{2}}^{4}+\|v\|_{L^{2}}^{2}\|\nabla u\|_{L^{2}}^{2}+
\|u\|_{L^{2}}^{2}\|\nabla u\|_{L^{2}}^{2}\right)\|\nabla q_{v}\|_{L^{2}}^{2}+\frac{1}{8}\|\Delta q_{v}\|_{L^{2}}^{2}.
 \end{align}
 Similarly,
\begin{align}\label{qvH1h2}
 &\int_{\mathcal{M}}w\partial_{p}q_{v}\cdot\Delta q_{v}d\mathcal{M}\leq\int_{\mathcal{M}'}\int_{p_{0}}^{p_{1}}|\nabla v|dp\int_{p_{0}}^{p_{1}}|\partial_{p}q_{v}||\Delta q_{v}|dpd\mathcal{M}'\nonumber\\
 \leq& C\|\Delta q_{v}\|_{L^{2}}\|\nabla v\|_{L^{2}}^{\frac{1}{2}}\left(
 \|\nabla v\|_{L^{2}}^{\frac{1}{2}}+\|\nabla^{2} v\|_{L^{2}}^{\frac{1}{2}}\right)\|\partial_{p}q_{v}\|_{L^{2}}^{\frac{1}{2}}\left(
 \|\partial_{p}q_{v}\|_{L^{2}}^{\frac{1}{2}}+\|\nabla \partial_{p}q_{v}\|_{L^{2}}^{\frac{1}{2}}\right)\\
 \leq& C\left(\|\nabla v\|_{L^{2}}^{2}+\|\nabla v\|_{L^{2}}^{4}+\|\Delta v\|_{L^{2}}^{2}+
 \|\nabla v\|_{L^{2}}^{2}\|\Delta v\|_{L^{2}}^{2}\right)\|\partial_{p}q_{v} \|_{L^{2}}^{2}+\frac{1}{8}\|\Delta q_{v}\|_{L^{2}}^{2}+\frac{3}{4}\|\nabla \partial_{p}v\|_{L^{2}}^{2}.\nonumber
 \end{align}
 Combining all the above inequalities, we can infer that
 \begin{align}\label{qvH1h3}
&\frac{d}{dt}\left(\frac{1}{2}\|\nabla q_{v}\|_{L^{2}}^{2}+
\alpha_{lv}\int_{\Gamma_{l}}\left(\frac{1}{2}(q_{v})^{2}-q_{v}q_{blv}\right)d\Gamma_{l}\right)
+\frac{5}{8}\|\Delta q_{v}\|_{L^{2}}^{2}+\frac{1}{4}\|\nabla\partial_{p}q_{v}\|_{w}^{2}\nonumber\\
\leq &A_{4}(t)(\|\nabla q_{v}\|_{L^{2}}^{2}+\|\partial_{p}q_{v} \|_{L^{2}}^{2})+A_{5}(t),
\end{align}
where
\begin{align*}
A_{4}(t)=&C\left(\|\nabla v\|_{L^{2}}^{2}+\|\nabla v\|_{L^{2}}^{4}+\|\Delta v\|_{L^{2}}^{2}+
 \|\nabla v\|_{L^{2}}^{2}\|\Delta v\|_{L^{2}}^{2}+\|v\|_{L^{2}}^{2}+\|u\|_{L^{2}}^{2}+\|\nabla u\|_{L^{2}}^{2}+\right.\nonumber\\
 &\left.\|v\|_{L^{2}}^{4}+\|\nabla v\|_{L^{2}}^{2}\|v\|_{L^{2}}^{2}+\|u\|_{L^{2}}^{2}\|\nabla v\|_{L^{2}}^{2}+\|u\|_{L^{2}}^{4}+\|v\|_{L^{2}}^{2}\|\nabla u\|_{L^{2}}^{2}+\|u\|_{L^{2}}^{2}\|\nabla u\|_{L^{2}}^{2}+1\right)\nonumber
\end{align*}
and
\begin{align*}
A_{5}(t)=C(1+\|q_{v}\|_{L_{2}}^{2}+\|\nabla v\|_{L^{2}}^{2}).
\end{align*}

By taking the inner product of equation $(\ref{ae3})$ with $-\partial_{p}^{2}q_{v}$ in the $L^{2}$ space, we have
\begin{align}\label{qvH10}
&-\int_{\mathcal{M}}\partial_{t}q_{v}\partial_{p}^{2}q_{v}d\mathcal{M}
-\int_{\mathcal{M}}\mathcal{A}_{q_{v}}q_{v}\partial_{p}^{2}q_{v}d\mathcal{M}
\nonumber\\
= &\int_{\mathcal{M}}\left(v\cdot\nabla q_{v}+w\partial_{p}q_{v}\right)\partial_{p}^{2}q_{v}d\mathcal{M}
-\int_{\mathcal{M}}\left(f_{q_{v}}-w^{-}Fh_{q_{v}}\right)\partial_{p}^{2}q_{v}d\mathcal{M}.
\end{align}
Through similar arguments as in \cite{TanLiu}, we can get that
\begin{align}
-\int_{\mathcal{M}}\partial_{t}q_{v}\partial_{p}^{2}q_{v}d\mathcal{M}
\geq&\frac{d}{dt}\left(\frac{1}{2}\|\partial_{p}q_{v}\|_{L^{2}}^{2}
+\beta_{v}\int_{\Gamma_{i}}\left(\frac{1}{2}(q_{v})^{2}-q_{v_{\ast}}q_{v}\right)d\Gamma_{i}\right)\nonumber\\
&-C\left(\|\partial_{p}q_{v}\|_{L^{2}}^{2}+\|q_{v}\|_{L^{2}}^{2}+1\right),
\end{align}
\begin{align}\label{qvH103}
-\int_{\mathcal{M}}\mathcal{A}_{q_{v}}q_{v}\partial_{p}^{2}q_{v}d\mathcal{M}
\geq \frac{3}{4}\left(\|\nabla\partial_{p}q_{v}\|_{L^{2}}^{2}
+\left\|\partial_{p}^{2}q_{v}\right\|_{w}^{2} \right)-C\left(\|\partial_{p}q_{v}\|_{L^{2}}^{2}+1\right),
\end{align}
and
\begin{align}\label{qvH1051}
\int_{\mathcal{M}}\left(f_{q_{v}}-{w}^{-}Fh_{q_{v}}\right)\partial_{p}^{2}q_{v}d\mathcal{M}
\leq C(1+\|q_{v}\|_{L^{2}}^{2}+\|\nabla v\|_{L^{2}}^{2})+\frac{1}{12}\|\partial_{p}^{2}q_{v}\|_{L^{2}}^{2}.
\end{align}
Then we consider the most problematic term
\begin{align*}
\int_{\mathcal{M}}\left(v\cdot\nabla q_{v}+w\partial_{p}q_{v}\right)\partial_{p}^{2}q_{v}d\mathcal{M}.
\end{align*}
In fact, through similar arguments as in (\ref{qvH1h1}) and (\ref{qvH1h2}), we have
\begin{align*}
 &\int_{\mathcal{M}}(v\cdot\nabla)q_{v}\cdot\partial_{p}^{2} q_{v}d\mathcal{M}\leq\int_{\mathcal{M}'} \int_{p_{0}}^{p_{1}} |v|+|\partial_{p}v|dp\int_{p_{0}}^{p_{1}} |\nabla q_{v}||\partial_{p}^{2} q_{v}|dp\mathcal{M}'\nonumber\\
\leq& C \left(\|v\|_{L^2}^{\frac{1}{2}}+\|u\|_{L^2}^{\frac{1}{2}}\right)
\left(\|v\|_{L^{2}}^{\frac{1}{2}} + \|\nabla v\|_{L^{2}}^{\frac{1}{2}}+\|u\|_{L^2}^{\frac{1}{2}}+\|\nabla u\|_{L^2}^{\frac{1}{2}}\right)\|\partial_{p}^{2} q_{v}\|_{L^{2}}\times\nonumber\\
&\ \ \  \|\nabla q_{v}\|_{L^{2}}^{\frac{1}{2}}\left(\|\nabla q_{v}\|_{L^{2}}^{\frac{1}{2}}+ \|\nabla^{2} q_{v}\|_{L^{2}}^{\frac{1}{2}}\right)\nonumber\\
\leq&C\left(\|v\|_{L^{2}}^{2}+\|u\|_{L^{2}}^{2}+\|\nabla v\|_{L^{2}}^{2}+\|\nabla u\|_{L^{2}}^{2}
+\|v\|_{L^{2}}^{4}+\|\nabla v\|_{L^{2}}^{2}\|v\|_{L^{2}}^{2}+\|u\|_{L^{2}}^{2}\|\nabla v\|_{L^{2}}^{2}\right.\nonumber\\
&\left.+\|u\|_{L^{2}}^{4}+\|v\|_{L^{2}}^{2}\|\nabla u\|_{L^{2}}^{2}+
\|u\|_{L^{2}}^{2}\|\nabla u\|_{L^{2}}^{2}\right)\|\nabla q_{v}\|_{L^{2}}^{2}+\frac{1}{8}\|\Delta q_{v}\|_{L^{2}}^{2}
+\frac{1}{12}\|\partial_{p}^{2}q_{v}\|_{L^{2}}^{2}.
 \end{align*}
 Similarly,
\begin{align*}
 &\int_{\mathcal{M}}w\partial_{p}q_{v}\partial_{p}^{2} q_{v}d\mathcal{M}\leq\int_{\mathcal{M}'}\int_{p_{0}}^{p_{1}}|\nabla v|dp\int_{p_{0}}^{p_{1}}|\partial_{p}q_{v}||\partial_{p}^{2} q_{v}|dpd\mathcal{M}'\nonumber\\
 \leq& C\|\partial_{p}^{2} q_{v}\|_{L^{2}}\|\nabla v\|_{L^{2}}^{\frac{1}{2}}\left(
 \|\nabla v\|_{L^{2}}^{\frac{1}{2}}+\|\nabla^{2} v\|_{L^{2}}^{\frac{1}{2}}\right)\|\partial_{p}q_{v}\|_{L^{2}}^{\frac{1}{2}}\left(
 \|\partial_{p}q_{v}\|_{L^{2}}^{\frac{1}{2}}+\|\nabla \partial_{p}q_{v}\|_{L^{2}}^{\frac{1}{2}}\right)\nonumber\\
 \leq& C\left(\|\nabla v\|_{L^{2}}^{2}+\|\nabla v\|_{L^{2}}^{4}+\|\Delta v\|_{L^{2}}^{2}+
 \|\nabla v\|_{L^{2}}^{2}\|\Delta v\|_{L^{2}}^{2}\right)\|\partial_{p}q_{v} \|_{L^{2}}^{2}+\frac{1}{12}\|\partial_{p}^{2} v\|_{L^{2}}^{2}+\frac{3}{4}\|\nabla \partial_{p}q_{v}\|_{L^{2}}^{2}.
 \end{align*}
 Thus,
 \begin{align}\label{qvH1p22}
 &\frac{d}{dt}\left(\frac{1}{2}\|\partial_{p}q_{v}\|_{L^{2}}^{2}
+\int_{\Gamma_{i}}\left(\frac{1}{2}(q_{v})^{2}-q_{v_{\ast}}q_{v}\right)d\Gamma_{i}\right)
+\frac{1}{4}\|\nabla\partial_{p}q_{v}\|_{L^{2}}^{2}
+\frac{1}{2}\left\|\partial_{p}^{2}q_{v}\right\|_{w}^{2}\nonumber\\
 \leq&A_{1}(t)(\|\nabla q_{v}\|_{L^{2}}^{2}+\|\partial_{p}q_{v} \|_{L^{2}}^{2})+A_{2}(t)+\frac{1}{8}\|\Delta q_{v}\|_{L^{2}}^{2}.
 \end{align}
 Combining (\ref{qvH1h3}) and (\ref{qvH1p22}), we can infer that
 \begin{align}
 &\frac{d}{dt}\left(\frac{1}{2}\|q_{v}\|_{H^{1}}^{2}
+\int_{\Gamma_{i}}\left(\frac{1}{2}(q_{v})^{2}-q_{v_{\ast}}q_{v}\right)d\Gamma_{i}
+
\int_{\Gamma_{l}}\left(\frac{1}{2}(q_{v})^{2}-q_{v}q_{blv}\right)d\Gamma_{l}\right)
+\frac{1}{2}\|q_{v}\|_{H^{2}}^{2}\nonumber\\
 \leq&A_{4}(t)\|q_{v}\|_{H^{1}}^{2}+A_{5}(t).
 \end{align}
 Utilizing the H\"older inequality and the Young inequality, we can obtain that
 \begin{align*}
G=&\int_{\Gamma_{i}}\left(\frac{1}{2}(q_{v})^{2}-q_{v_{\ast}}q_{v}\right)d\Gamma_{i}
+\int_{\Gamma_{l}}\left(\frac{1}{2}(q_{v})^{2}-q_{v}q_{blv}\right)d\Gamma_{l}\nonumber\\
\geq
 &-C(\|q_{v_{\ast}}\|_{L^{2}(\Gamma_{i})}^{2}+\|q_{blv}\|_{L^{2}(\Gamma_{l})}^{2}):=-C_{\ast}.
 \end{align*}
 Thus
 \begin{align*}
 \frac{d}{dt}\left(\frac{1}{2}\|q_{v}\|_{H^{1}}^{2}+G+C_{\ast}\right)
+\frac{1}{2}\|q_{v}\|_{H^{2}}^{2} \leq A_{4}(t)(G+C_{\ast}+\|q_{v}\|_{H^{1}}^{2})+A_{5}(t).
 \end{align*}
 Considering regularities of $v,u,q_{v}$ in (\ref{L2-v-theta-qj}) and (\ref{uL2}) respectively, we know the $L^{2}$ integrability of $A_{4}$ and $A_{5}$. Then using the Gronwall inequality, we can complete the proof of this lemma.
 \end{proof}

 Thus, we get the following result about the local existence of strong solutions:
 \begin{proposition}
Let $v_{0}\in\mathbb{V}$, $\theta_{0}, q_{v0}, q_{r0},q_{c0}\in H^{1}$. Then there exists $t_{\ast}\in (0,t_{1}]$ and a solution $(v,\theta,q_{v},q_{c},q_{r})$ to equations (\ref{e1})-(\ref{e7}) with boundary condition (\ref{boundary condition}) and initial condition (\ref{initial condition}) satisfying that
\begin{align}
\sup_{0\leq t\leq t_{\ast}}\|(v,\theta,q_{v},q_{c},q_{r})\|_{H^{1}}^{2}
+\int_{0}^{t_{\ast}}\|\Delta v\|_{L^{2}}^{2}+\|\nabla\partial_{p}v\|_{L^{2}}^{2}+\|(\theta,q_{v},q_{c},q_{r})\|_{H^{2}}^{2}dt\leq \mathcal{C}_{1},
\end{align}
where $\mathcal{C}_{1}$ depends on the initial data and $t_{\ast}$.
\end{proposition}

\subsection{The time regularity of strong solutions}
In this subsection, we consider the time regularity of strong solutions. In fact, we have obtained the time regularity of velocity field $v$ in space $L^{2}(0,t_{\ast};\mathbb{H})$. Thus we mainly talk about the time regularities of $\theta,q_{v},q_{c},q_{r}$. Here we give a detailed proof of the time regularity of $q_{c}$. The proof of time regularity of other variables can be obtained similarly.
\begin{lemma}\label{H1t}
Let $v_{0}\in\mathbb{V}$, $\theta_{0}, q_{v0}, q_{r0},q_{c0}\in H^{1}$, and $q_{c}$ be the solution to equation (\ref{e6}) with boundary condition (\ref{boundary condition}) and initial condition (\ref{initial condition}). Then
\begin{align*}
\partial_{t}q_{c}\in L^{2}(0,t_{\ast};L^{2}).
\end{align*}
\end{lemma}
\begin{proof}
 Choosing a test function $\tilde{q}_{c}\in L^{2}(0,t_{\ast};L^{2})$ with $\|\tilde{q}_{c}\|_{L^{2}(0,t_{\ast};L^{2})}\leq 1$, we can infer from equation (\ref{e6}) that
 \begin{align}\label{qct1}
|\langle\partial_{t}q_{c},\tilde{q}_{c}\rangle|\leq
  |\langle\mathcal{A}_{q_{c}}q_{c},\tilde{q}_{c}\rangle|+ |\langle v\cdot\nabla q_{c}+w\partial_{p}q_{c},\tilde{q}_{c}\rangle|
 +|\langle f_{q_{c}}+w^{-}Fh_{q_{v}},\tilde{q}_{c}\rangle|.
 \end{align}
 Using the H\"older inequality, we obtain that
  \begin{align}\label{qct2}
 |\langle\mathcal{A}_{q_{c}}q_{c},\tilde{q}_{c}\rangle|\leq C\|q_{c}\|_{H^{2}}\|\tilde{q}_{c}\|_{L^{2}}.
 \end{align}
By the inequality in Lemma \ref{trilinear term lemma}, we can infer that
 \begin{align}\label{qct3}
 &|\langle(v\cdot\nabla)q_{c},\tilde q_{c}\rangle|\leq\int_{\mathcal{M}'} \int_{p_{0}}^{p_{1}} |v|+|\partial_{p}v|dp\int_{p_{0}}^{p_{1}} |\nabla q_{c}||\tilde q_{c}|dp\mathcal{M}'\nonumber\\
\leq &C \left(\|v\|_{L^2}^{\frac{1}{2}}+\|u\|_{L^2}^{\frac{1}{2}}\right)
\left(\|v\|_{L^{2}}^{\frac{1}{2}} + \|\nabla v\|_{L^{2}}^{\frac{1}{2}}+\|u\|_{L^2}^{\frac{1}{2}}+\|\nabla u\|_{L^2}^{\frac{1}{2}}\right)\|\tilde q_{c}\|_{L^{2}}\|\nabla q_{c}\|_{L^{2}}^{\frac{1}{2}}\cdot\nonumber\\
&\ \ \left(\|\nabla q_{c}\|_{L^{2}}^{\frac{1}{2}}+ \|\nabla^{2} q_{c}\|_{L^{2}}^{\frac{1}{2}}\right)\nonumber\\
\leq &C\left(\|v\|_{L^{2}}\|\nabla q_{c}\|_{L^{2}}+\|\nabla v\|_{L^{2}}\|\nabla q_{c}\|_{L^{2}}+\|u\|_{L^{2}}\|\nabla q_{c}\|_{L^{2}}+\|\nabla u\|_{L^{2}}\|\nabla q_{c}\|_{L^{2}}+\|v\|_{L^{2}}\|\Delta q_{c}\|_{L^{2}}\right.\nonumber\\
&\left.+\|\nabla v\|_{L^{2}}\|\Delta q_{c}\|_{L^{2}}+\|u\|_{L^{2}}\|\Delta q_{c}\|_{L^{2}}\right)\|\tilde q_{c}\|_{L^{2}}.
 \end{align}
 Similarly,
\begin{align}\label{qct4}
 &|\langle w\partial_{p}q_{c},\tilde q_{c}\rangle|\leq\int_{\mathcal{M}'}\int_{p_{0}}^{p_{1}}|\nabla v|dp\int_{p_{0}}^{p_{1}}|\partial_{p}q_{c}||\tilde q_{c}|dpd\mathcal{M}'\nonumber\\
\leq &C\|\tilde q_{c}\|_{L^{2}}\|\nabla v\|_{L^{2}}^{\frac{1}{2}}\left(
 \|\nabla v\|_{L^{2}}^{\frac{1}{2}}+\|\nabla^{2} v\|_{L^{2}}^{\frac{1}{2}}\right)\|\partial_{p}q_{c}\|_{L^{2}}^{\frac{1}{2}}\left(
 \|\partial_{p}q_{c}\|_{L^{2}}^{\frac{1}{2}}+\|\nabla \partial_{p}q_{c}\|_{L^{2}}^{\frac{1}{2}}\right)\\
\leq &C\left(\|\nabla v\|_{L^{2}}\|\partial_{p}q_{c}\|_{L^{2}}+\|\nabla v\|_{L^{2}}\|\nabla\partial_{p}q_{c}\|_{L^{2}}+\|\Delta v\|_{L^{2}}\|\partial_{p}q_{c}\|_{L^{2}}+
 \|\nabla v\|_{L^{2}}\|\Delta v\|_{L^{2}}\right.\nonumber\\
 &\left.+\|\partial_{p}q_{c}\|_{L^{2}}\|\nabla\partial_{p}q_{c}\|_{L^{2}}
 \right)\|\tilde q_{c}\|_{L^{2}}^{2}.
 \end{align}
 Recalling definitions of $F$ and $f_{q_{c}}$, we have
 \begin{align}\label{qct5}
 \left|\langle f_{q_{c}}
+w^{-}Fh_{q_{v}},\tilde{q}_{c}\rangle\right|
\leq &\left\|-k_{1}(q_{c}-q_{crit})^{+}-k_{2}q_{c}\tau(q_{r})+w^{-}Fh_{q_{v}}\right\|_{L^{2}}\|\tilde{q}_{c}\|_{L^{2}}
\nonumber\\
\leq &
C (1+\|q_{c}\|_{L^{2}}+\|\nabla v\|_{L^{2}})\|\tilde{q}_{c}\|_{L^{2}},
 \end{align}
 where we have used the uniform boundness of $\tau(q_{r})$ in the last step.

  Combining (\ref{qct1})-(\ref{qct5}), we can obtain that
 \begin{align*}
\|\partial_{t}q_{c}\|_{L_{2}}
\leq A_{6}(t)
\end{align*}
where
\begin{align*}
 A_{6}(t)&=C\left(1+\|q_{c}\|_{H^{2}}+\|q_{c}\|_{L^{2}}+\|\nabla v\|_{L^{2}}+\|v\|_{L^{2}}\|\nabla q_{c}\|_{L^{2}}+\|\nabla v\|_{L^{2}}\|\nabla q_{c}\|_{L^{2}}+\|u\|_{L^{2}}\|\nabla q_{c}\|_{L^{2}}\right.\nonumber\\
  &\left.+\|\nabla u\|_{L^{2}}\|\nabla q_{c}\|_{L^{2}}+\|v\|_{L^{2}}\|\Delta q_{c}\|_{L^{2}}
+\|\nabla v\|_{L^{2}}\|\Delta q_{c}\|_{L^{2}}+\|u\|_{L^{2}}\|\Delta q_{c}\|_{L^{2}}+\|\nabla v\|_{L^{2}}\|\partial_{p}q_{c}\|_{L^{2}}\right.\nonumber\\
&\left.+\|\nabla v\|_{L^{2}}\|\nabla\partial_{p}q_{c}\|_{L^{2}}+
\|\Delta v\|_{L^{2}}\|\partial_{p}q_{c}\|_{L^{2}}+
 \|\nabla v\|_{L^{2}}\|\Delta v\|_{L^{2}}
+\|\partial_{p}q_{c}\|_{L^{2}}\|\nabla\partial_{p}q_{c}\|_{L^{2}}\right).
 \end{align*}
 Considering regularities for $v$ and $q_{c}$, it is obviously that $A_{6}(t)$ is $L^{2}$ integrable on $(0,t_{\ast})$. Then integrating in time from $0$ to $t_{\ast}$, we can infer that
 \begin{align*}
 \int_{0}^{t_{\ast}}\|\partial_{t}q_{c}\|_{L_{2}}^{2}dt
 \leq &\int_{0}^{t_{\ast}}A_{6}^{2}(t)dt<\infty,
 \end{align*}
 which shows that $q_{c}\in L^{2}(0,t_{\ast};L^{2})$.
 \end{proof}
\subsection{The global existence of strong solutions}
Checking the proof in previous sections, we find that the arising of $t_{\ast}$ is due to the $L^{2}$ a priori estimate for $\partial_{p}v$. If we want to obtain the global existence of strong solutions, the $L^{\infty}(0,t_{1};L^{2}(\mathcal{M}))$ estimate for $\partial_{p}v$ should rebuilt.  In fact, through a similar argument as in \cite{CaoTiti6} and \cite{Hussein}, we know that, if $\|v_{0}\|_{H^{1}}+\left\|\partial_{p}v_{0}\right\|_{L^{m}}+
\|v_{0}\|_{L^{\infty}}<\infty$ for some $m>2$, then
for any time $0\leq t\leq t_{1}$,
\begin{align}\label{Global-v}
\sup\limits_{s\in [0,t]}\left(\|v\|_{H^{1}}^{2}
+\left\|\partial_{p}v\right\|_{L^{m}}^{m}\right)+
\int_{0}^{t}\|\Delta v\|_{L^{2}}^{2}+\big\|\nabla\partial_{p}v\big\|_{w}^{2}ds
\leq C,
\end{align}
where $C$ is a constant depending on $m,t$ and $\|v_{0}\|_{H^{1}}+\|\partial_{p}v_{0}\|_{L^{m}}+
\|v_{0}\|_{L^{\infty}}$.

Combining the $L^{2}$ estimates for $\theta,q_{v},q_{c},q_{r}$ in Lemma \ref{l2prior} and the $H^{1}$ estimate for $v$ in (\ref{Global-v}), and testifying their time regularities as in Lemma \ref{trv}, we can get the global existence of quasi-strong solutions.
\begin{proposition}\label{Globalquasi-strong}
Let $(\theta_{0},q_{v},q_{c},q_{r})\in (L^{2})^{4}$, $v_{0}\in L^{\infty}\cap\mathbb{V}$, $\partial_{p}v_{0}\in L^{m}$ for some $m>2$,
 $\|v_{0}\|_{H^{1}}+\|\partial_{p}v_{0}\|_{L^{m}}+
\|v_{0}\|_{L^{\infty}}<\infty$. Then there exists a global quasi-strong solution $(v,\theta,q_{v},q_{c},q_{r})$ to equations (\ref{e1})-(\ref{e7}) with boundary condition (\ref{boundary condition}) and initial condition (\ref{initial condition}). Moreover
\begin{align*}
\sup_{0\leq t\leq t_{1}}\left(\|v\|_{\mathbb{V}}^{2}+\|(\theta,q_{j})\|_{L^{2}}^{2}+\big\|\partial_{p}v\big\|_{L^{m}}^{m}\right)
+\int_{0}^{t_{1}}\|\Delta v\|_{L^{2}}^{2}+\|\nabla\partial_{p} v\|_{w}^{2}+\|(\theta,q_{j})\|_{H^{1}}^{2}dt\leq \mathcal{C}_{2},
\end{align*}
for $j\in\{v,c,r\}$, where $\mathcal{C}_{2}$ depends on the initial data, $m$ and $t_{1}$.
\end{proposition}

With the uniform estimate \eqref{Global-v} for $v$ at hand,  proceeding exactly as in the proof of Lemma \ref{H1}
to seek uniform $H^{1}$-estimates for $\theta$ and $q_{v},q_{c},q_{r}$, we are able to obtain that, for $j\in\{v,c,r\}$
\begin{align*}
\sup\limits_{t\in [0,t_{1}]}\left(\|(v,\theta,q_{j})\|_{H^{1}}^{2}
+\big\|\partial_{p}v\big\|_{L^{m}}^{m}\right)+
\int_{0}^{t_{1}}\|\Delta v\|_{L^{2}}^{2}+\|\nabla\partial_{p}v\big\|_{w}^{2}+\|(\theta,q_{j})\|_{H^{2}}^{2}ds
\leq C,
\end{align*}
where $C$ is a constant depending on $m,t_{1}$, $\|(\theta_{0},q_{j0})\|_{H^{1}}$ as well as $\|v_{0}\|_{H^{1}}+\|\partial_{p}v_{0}\|_{L^{m}}+
\|v_{0}\|_{L^{\infty}}$.
At the same time, following the proof in Lemma \ref{H1t}, we can also obtain that
\begin{align*}
\left(\partial_{t}v,\partial_{t}\theta,\partial_{t}q_{v},\partial_{t}q_{c},\partial_{t}q_{r}\right)\in L^{2}(0,t_{1};\mathbb{H}\times L^{2}(\mathcal{M})^{4}).
\end{align*}
Then using the Aubin-lions compactness theorem, we deduce that
\begin{align*}
\left(v,\theta,q_{v},q_{c},q_{r}\right)\in C(0,t_{1};\mathbb{V}\times(H^{1})^{4}).
\end{align*}

As summary, we get the global existence of strong solutions.
 \begin{proposition}\label{Global whole}
Let $(v_{0},\theta_{0},q_{v},q_{c},q_{r})\in \mathbb{V}\times (H^{1})^{4}$, $v_{0}\in L^{\infty}$, $\partial_{p}v_{0}\in L^{m}$ for some $m>2$,
 $\|v_{0}\|_{H^{1}}+\|\partial_{p}v_{0}\|_{L^{m}}+
\|v_{0}\|_{L^{\infty}}<\infty$. Then there exists a global strong solution $(v,\theta,q_{v},q_{c},q_{r})$ to equations (\ref{e1})-(\ref{e7}) with boundary condition (\ref{boundary condition}) and initial condition (\ref{initial condition}). Moreover
\begin{align*}
\sup_{0\leq t\leq t_{1}}\left(\|(v,\theta,q_{v},q_{c},q_{r})\|_{H^{1}}^{2}+\big\|\partial_{p}v\big\|_{L^{m}}^{m}\right)
+\int_{0}^{t_{1}}\|\Delta v\|_{L^{2}}^{2}+\|\nabla\partial_{p}v\|_{L^{2}}^{2}+\|(\theta,q_{v},q_{c},q_{r})\|_{H^{2}}^{2}dt\leq \mathcal{C}_{3},
\end{align*}
where $\mathcal{C}_{3}$ depends on the initial data, $m$ and $t_{1}$.
\end{proposition}

\section{The uniqueness of solutions}
In this section, we prove the uniqueness of quasi-strong solution to equations (\ref{e1})-(\ref{e7}). Then the uniqueness of strong solution naturally holds. In this section, we return to consider the $(v,T,q_{v},q_{c},q_{r})$ system. Considering the relationship between $T$ and $\theta$, we know that the existence results of quasi-strong and strong solution for $\theta$ still hold for $T$. In order to overcome the difficulty caused by the Heaviside function, two new unknown quantities are introduced to substitute $T$ and $q_{c}$ as in \cite{Hittmeir,Hittmeir2017}, while the monotone operator theory is used in dealing with $q_{v}$ as in \cite{Zelati}.

 Recalling the relationship between $\theta$ and $T$, we suppose that $T$ satisfies the following boundary conditions:
\begin{align*}
{\rm on}\ \Gamma_{i}:\ \partial_{p}T=T_{\ast}-T,\
 {\rm on}\ \Gamma_{u}:\ \partial_{p}T=0,\
 {\rm on}\ \Gamma_{l}:\ \partial_{\textbf{n}}T=T_{bl}-T.
\end{align*}
where $T_{\ast}$ and $T_{bl}$ are given sufficiently smooth temperature distribution.

Combining the Stampacchia method and De Giorgi iterations as in \cite{Zelati} and \cite {Hittmeir2017}, we can get the following uniform boundness result:
\begin{lemma}\label{unform boundness}
Let $(v_{0},T_{0},q_{v0},q_{c0},q_{r0})\in\mathbb{V}\times (H^{1})^{4}\cap L^{\infty}
(\mathcal{M})^{6}$, $\|v_{0}\|_{H^{1}}+\|\partial_{p}v_{0}\|_{L^{m}}+
\|v_{0}\|_{L^{\infty}}<\infty$ for some $m>2$ and $T_{0},q_{v0},q_{c0},q_{r0}$ be nonnegative initial data.  Then for any $t\in[0,t_{1}]$,
\begin{align}
0\leq q_{v}\leq q_{v}^{\ast},\ \ 0\leq q_{c}\leq q_{c}^{\ast},\ \ 0\leq q_{r}\leq q_{r}^{\ast},\ \ 0\leq T\leq T^{\ast},
\end{align}
where
\begin{align*}
q_{v}^{\ast}&={\rm max}\{\|q_{v0}\|_{L^{\infty}},\|q_{v\ast}\|_{L^{\infty}(0,t;\Gamma_{i})},
\|q_{blv}\|_{L^{\infty}(0,t;\Gamma_{l})}\},\nonumber\\
q_{c}^{\ast}&={\rm max}\{\|q_{c0}\|_{L^{\infty}},\|q_{c\ast}\|_{L^{\infty}(0,t;\Gamma_{i})},
\|q_{blc}\|_{L^{\infty}(0,t;\Gamma_{l})}\},\nonumber\\
q_{r}^{\ast}&={\rm max}\{\|q_{r0}\|_{L^{\infty}},\|q_{r\ast}\|_{L^{\infty}(0,t;\Gamma_{i})},
\|q_{blr}\|_{L^{\infty}(0,t;\Gamma_{l})}\},\nonumber\\
T^{\ast}&={\rm max}\{\|T_{0}\|_{L^{\infty}},\|T_{\ast}\|_{L^{\infty}(0,t;\Gamma_{i})},
\|T_{bl}\|_{L^{\infty}(0,t;\Gamma_{l})}\}.
\end{align*}
\end{lemma}
\begin{remark}
During the De Giorgi iterations process as in \cite{Zelati}, one should note that
$\|w\|_{H^{1}}^{2}$ can be bounded by $\|\Delta v\|_{L^{2}}^{2}+\|\nabla\partial_{p}v\|_{L^{2}}^{2}$, whose $L^{2}$ integrability in time interval $[0,t_{1}]$ can be ensured by the result in Proposition \ref{Global whole}.
\end{remark}

\begin{proposition}\label{uniqueness}
Assume $(v_{1},T_{1},q_{v1},q_{c1},q_{r1})$ and  $(v_{2},T_{2},q_{v2},q_{c2},q_{r2})$ are two quasi-strong solutions to equations (\ref{e1})-(\ref{e7}) corresponding to the same initial data $(v_{0},T_{0},q_{v0},q_{c0},q_{r0})$, with the function $F$ in source terms replaced by its positive part $F^{+}$. Then
\begin{align*}
v_{1}=v_{2},T_{1}=T_{2},q_{v1}=q_{v2},q_{c1}=q_{c2},q_{r1}=q_{r2},
\end{align*}
in the sense of $L^{2}$.
\end{proposition}
\begin{remark}
It is worth noting that the uniqueness holds under the assumption that the function $F$ is replaced by its positive part $F^{+}$. This assumption is in line with physical reality. For more detailed explanation, we refer readers to \cite{Zelati}. In addition, due to the uniform boundness results in Lemma \ref{unform boundness}, we can take $\tau(q_{r})=q_{r}$ in this section.
\end{remark}
\begin{proof}
Let $(v_{1},T_{1},q_{v1},q_{c1},q_{r1})$ and $(v_{2},T_{2},q_{v2},q_{c2},q_{r2})$ be two global quasi-strong solutions corresponding to initial data $(v_{1}^{0},T_{1}^{0},q_{v1}^{0},q_{c1}^{0},q_{r1}^{0})$ and $(v_{2}^{0},T_{2}^{0},q_{v2}^{0},q_{c2}^{0},q_{r2}^{0})$ respectively.

Set
\begin{align*}
\hat{v}=v_{1}-v_{2},\ \hat{T}=T_{1}-T_{2},\ \hat{q}_{j}=q_{j1}-q_{j2}, j\in\{v,c,r\}.
\end{align*}
Then $\hat{v}$ satisfies that
\begin{align}\label{hat{v0}}
\partial_{t}\hat{v}-\Delta\hat{v}+(v_{1}\cdot\nabla)\hat{v}+w_{1}\partial_{p}\hat{v}+
(\hat{v}\cdot\nabla)v_{2}+\hat{w}\partial_{p}v_{2}+f\hat{v}^{\bot}+\nabla\hat{\Phi}_{s}+\nabla\int_{p}^{p_{1}}\frac{R}{p'}\hat{T}dp'=0,
\end{align}
where $\hat{w}=w_{1}-w_{2}$ and $\hat{\Phi}_{s}=\Phi_{s1}-\Phi_{s2}$. And the corresponding boundary conditions are
 \begin{align}
 {\rm on}\ \Gamma_{i}:\ \partial_{p}\hat{v}=0,\ \hat{w}=0,\
 {\rm on}\ \Gamma_{u}:\ \partial_{p}\hat{v}=0,\ \hat{w}=0,\ \partial_{p}\hat{w}=0,\
 {\rm on}\ \Gamma_{l}:\ \hat{v}=0,\ \partial_{\textbf{n}}\hat{v}=0.\nonumber
 \end{align}

Taking the inner product of equation (\ref{hat{v0}}) with $\hat{v}$ in $L^{2}$ space, using integration by parts, we can deduce that
\begin{align*}
\frac{1}{2}\frac{d}{dt}\|\hat{v}\|_{L^{2}}^{2}+\|\nabla \hat{v}\|_{L^{2}}^{2}
=&-\int_{\mathcal{M}}\left[(\hat{v}\cdot\nabla)v_{2}+\hat{w}\partial_{p}v_{2}\right]\cdot\hat{v}d\mathcal{M}
-\int_{\mathcal{M}}\left[(v_{1}\cdot\nabla)\hat{v}+w_{1}\partial_{p}\hat{v}\right]\cdot\hat{v}d\mathcal{M}\nonumber\\
&-\int_{\mathcal{M}}f\hat{v}^{\bot}\cdot\hat{v}d\mathcal{M}-
\int_{\mathcal{M}}\nabla\hat{\Phi}_{s}\cdot\hat{v}d\mathcal{M}
-\int_{\mathcal{M}}\nabla\int_{p}^{p_{1}}\frac{R}{p'}\hat{T}dp'\cdot\hat{v}d\mathcal{M}.
\end{align*}
It is easy to verify that
\begin{align*}
\int_{\mathcal{M}}f\hat{v}^{\bot}\cdot\hat{v}d\mathcal{M}=0.
\end{align*}
By integration by parts, we can obtain that
\begin{align*}
\int_{\mathcal{M}}\nabla\hat{\Phi}_{s}\cdot\hat{v}d\mathcal{M}=0,
\end{align*}
and
\begin{align*}
\int_{\mathcal{M}}\left[(v_{1}\cdot\nabla)\hat{v}+w_{1}\partial_{p}\hat{v}\right]\cdot\hat{v}d\mathcal{M}=0.
\end{align*}
Using integration by parts, the H\"older inequality and the Young inequality, we can infer that
\begin{align*}
\int_{\mathcal{M}}\nabla\int_{p}^{p_{1}}\frac{R}{p'}\hat{T}dp'\cdot\hat{v}d\mathcal{M}
\leq C\|\hat{T}\|_{L^{2}}^{2}+\frac{1}{8}\|\nabla\hat{v}\|_{L^{2}}^{2}.
\end{align*}
Utilizing the inequality in Lemma \ref{HHP}, we can get that
\begin{align*}
\int_{\mathcal{M}}(\hat{v}\cdot\nabla)v_{2}\cdot\hat{v}d\mathcal{M}
\leq& C\|\hat{v}\|_{L^{2}}\|\nabla\hat{v}\|_{L^{2}}(\|\nabla v_{2}\|_{L^{2}}+\|\nabla v_{2}\|_{L^{2}}^{\frac{1}{2}}\|\nabla\partial_{p}v_{2}\|_{L^{2}}^{\frac{1}{2}})\nonumber\\
\leq& C(\|\nabla v_{2}\|_{L^{2}}^{2}+\|\nabla\partial_{p}v_{2}\|_{L^{2}}^{2})\|\hat{v}\|_{L^{2}}^{2}
+\frac{1}{8}\|\nabla\hat{v}\|_{L^{2}}^{2}.
\end{align*}
Considering the inequality in Lemma \ref{trilinear term lemma}, we have
\begin{align*}
&\int_{\mathcal{M}}\hat{w}\partial_{p}v_{2}\cdot\hat{v}d\mathcal{M}
\leq\int_{\mathcal{M}'}\int_{p_{0}}^{p_{1}}|\nabla\hat{v}|dp
\int_{p_{0}}^{p_{1}}|\partial_{p}v_{2}||\hat{v}|dpd\mathcal{M}'\nonumber\\
\leq& C\|\nabla\hat{v}\|_{L^{2}}\|\partial_{p}v_{2}\|_{L^{2}}^{\frac{1}{2}}
(\|\partial_{p}v_{2}\|_{L^{2}}^{\frac{1}{2}}+\|\nabla\partial_{p}v_{2}\|_{L^{2}}^{\frac{1}{2}})
\|\hat{v}\|_{L^{2}}^{\frac{1}{2}}(\|\hat{v}\|_{L^{2}}^{\frac{1}{2}}+\|\nabla\hat{v}\|_{L^{2}}^{\frac{1}{2}})\nonumber\\
\leq& C\left(\|\partial_{p}v_{2}\|_{L^{2}}^{2}+\|\partial_{p}v_{2}\|_{L^{2}}^{4}+\|\nabla\partial_{p}v_{2}\|_{L^{2}}^{2}
+\|\partial_{p}v_{2}\|_{L^{2}}^{2}\|\nabla\partial_{p}v_{2}\|_{L^{2}}^{2} \right)\|\hat{v}\|_{L^{2}}^{2}
+\frac{1}{8}\|\nabla\hat{v}\|_{L^{2}}^{2}.
\end{align*}
Thus combining all the above inequalities, we can deduce that
\begin{align}\label{hat{v}}
\frac{d}{dt}\|\hat{v}\|_{L^{2}}^{2}+2\|\nabla \hat{v}\|_{L^{2}}^{2}
\leq &C\left(\|\nabla v_{2}\|_{L^{2}}^{2}+\|\nabla\partial_{p}v_{2}\|_{L^{2}}^{2}+\|\partial_{p}v_{2}\|_{L^{2}}^{2}
+\|\partial_{p}v_{2}\|_{L^{2}}^{4}+\|\nabla\partial_{p}v_{2}\|_{L^{2}}^{2}\right.\nonumber\\ &\left.
+\|\partial_{p}v_{2}\|_{L^{2}}^{2}\|\nabla\partial_{p}v_{2}\|_{L^{2}}^{2} \right)\|\hat{v}\|_{L^{2}}^{2}+C\|\hat{T}\|_{L^{2}}^{2}+\frac{3}{4}\|\nabla\hat{v}\|_{L^{2}}^{2}.
\end{align}

Set
\begin{align*}
Q=q_{v}+q_{c},\ \ H=T+\frac{L}{c_{p}\Pi}q_{v}.
\end{align*}
Then $\hat{Q}=Q_{1}-Q_{2}$ satisfies that
\begin{align}\label{hat{Q0}}
\partial_{t}\hat{Q}+v_{1}\cdot\nabla\hat{Q}+w_{1}\partial_{p}\hat{Q}+\hat{v}\cdot\nabla Q_{2}+\hat{w}\partial_{p} Q_{2}+A_{q}\hat{Q}
=f_{q_{v1}}-f_{q_{v2}}+f_{q_{c1}}-f_{q_{c2}},
\end{align}
with the boundary conditions
\begin{align*}
{\rm on}\ \Gamma_{i}:\ \partial_{p}\hat{Q}=-\hat{Q},\
 {\rm on}\ \Gamma_{u}:\ \partial_{p}\hat{Q}=0, \
 {\rm on}\ \Gamma_{l}:\ \partial_{\textbf{n}}\hat{Q}=-\hat{Q}.
\end{align*}
Taking the inner product of equation (\ref{hat{Q0}}) with $\hat{Q}$ in $L^{2}$ space, we can deduce that
\begin{align*}
&\frac{1}{2}\frac{d}{dt}\|\hat{Q}\|_{L^{2}}^{2}+\int_{\mathcal{M}}\hat{Q}A_{q}\hat{Q}d\mathcal{M}
+\int_{\mathcal{M}}\left[\hat{v}\cdot\nabla Q_{2}+\hat{w}\partial_{p} Q_{2}\right]\hat{Q}d\mathcal{M}\nonumber\\
=&\int_{\mathcal{M}}\left(f_{q_{v1}}-f_{q_{v2}}\right)\hat{Q}d\mathcal{M}
+\int_{\mathcal{M}}\left(f_{q_{c1}}-f_{q_{c2}}\right)\hat{Q}d\mathcal{M},
\end{align*}
where we have used
\begin{align*}
\int_{\mathcal{M}}(v_{1}\cdot\nabla\hat{Q}+w_{1}\partial_{p}\hat{Q})\hat{Q}d\mathcal{M}=0.
\end{align*}
Through integration by parts and considering the uniform boundness of $q_{v},q_{c}$, we can infer that
\begin{align}\label{multiterm}
\int_{\mathcal{M}}\left[\hat{v}\cdot\nabla Q_{2}+\hat{w}\partial_{p} Q_{2}\right]\hat{Q}d\mathcal{M}
=&-\int_{\mathcal{M}}\left[\hat{v}\cdot\nabla \hat{Q}+\hat{w}\partial_{p}\hat{Q}\right]Q_{2}d\mathcal{M}\nonumber\\
\leq&\|Q_{2}\|_{L^{\infty}}\|\hat{v}\|_{L^{2}}\|\nabla\hat{Q}\|_{L^{2}}+
\|Q_{2}\|_{L^{\infty}}\|\hat{w}\|_{L^{2}}\|\partial_{p}\hat{Q}\|_{L^{2}}\nonumber\\
\leq& C\|\hat{v}\|_{L^{2}}^{2}+\frac{1}{8}\|\nabla\hat{Q}\|_{L^{2}}^{2}+\frac{1}{8}\|\partial_{p}\hat{Q}\|_{L^{2}}^{2}
+\frac{1}{8\delta^{2}}\|\nabla\hat{v}\|_{L^{2}}^{2},
\end{align}
where $\delta$ is a sufficiently small fixed constant.
Recalling the definition of $f_{q_{v}}$, considering the uniform boundness of $q_{r}$ and $q_{v}$, we can infer that
\begin{align*}
&\int_{\mathcal{M}}\left(f_{q_{v1}}-f_{q_{v2}}\right)\hat{Q}d\mathcal{M}
=\int_{\mathcal{M}}\left[k_{3}q_{r1}(q_{vs}-q_{v1})^{+}-k_{3}q_{r2}(q_{vs}-q_{v2})^{+}\right]\hat{Q}d\mathcal{M}\nonumber\\
=&k_{3}\int_{\mathcal{M}}(q_{r1}-q_{r2})(q_{vs}-q_{v1})^{+}\hat{Q}d\mathcal{M}
+k_{3}\int_{\mathcal{M}}q_{r2}\left[(q_{vs}-q_{v1})^{+}-(q_{vs}-q_{v2})^{+}\right]\hat{Q}d\mathcal{M}\nonumber\\
\leq& C\|q_{r1}-q_{r2}\|_{L^{2}}\|\hat{Q}\|_{L^{2}}+C\|q_{v1}-q_{v2}\|_{L^{2}}\|\hat{Q}\|_{L^{2}}\leq C(\|\hat{q}_{r}\|_{L^{2}}^{2}+\|\hat{q}_{v}\|_{L^{2}}^{2}+\|\hat{Q}\|_{L^{2}}^{2}).
\end{align*}
Similarly,
\begin{align*}
&\int_{\mathcal{M}}\left(f_{q_{c1}}-f_{q_{c2}}\right)\hat{Q}d\mathcal{M}\nonumber\\
=&\int_{\mathcal{M}}\left(-k_{1}(q_{c1}-q_{crit})^{+}-k_{2}q_{c1}q_{r1}+k_{1}(q_{c2}-q_{crit})^{+}+k_{2}q_{c2}q_{r2}\right)\hat{Q}d\mathcal{M}\nonumber\\
=&-k_{1}\int_{\mathcal{M}}\left[(q_{c1}-q_{crit})^{+}-(q_{c2}-q_{crit})^{+}\right]\hat{Q}d\mathcal{M}
-k_{2}\int_{\mathcal{M}}\left(q_{c1}q_{r1}-q_{c2}q_{r2}\right)\hat{Q}d\mathcal{M}\nonumber\\
\leq& C\|q_{c1}-q_{c2}\|_{L^{2}}\|\hat{Q}\|_{L^{2}}+C\|q_{r1}-q_{r2}\|_{L^{2}}\|\hat{Q}\|_{L^{2}}\leq C(\|\hat{q}_{r}\|_{L^{2}}^{2}+\|\hat{q}_{c}\|_{L^{2}}^{2}+\|\hat{Q}\|_{L^{2}}^{2})\nonumber\\
\leq& C(\|\hat{q}_{v}\|_{L^{2}}^{2}+\|\hat{q}_{r}\|_{L^{2}}^{2}+\|\hat{Q}\|_{L^{2}}^{2}).
\end{align*}
Integrating by parts and considering the boundary condition of $\hat{Q}$, we have
\begin{align*}
\int_{\mathcal{M}}\hat{Q}A_{q}\hat{Q}d\mathcal{M}=\|\nabla\hat{Q}\|_{L^{2}}^{2}
+\|\partial_{p}\hat{Q}\|_{w}^{2}-\int_{\Gamma_{l}}\hat{Q}\partial_{\textbf{n}}\hat{Q}d\Gamma_{l}
-\int_{\Gamma_{i}}\hat{Q}\partial_{p}\hat{Q}d\Gamma_{i}\geq\|\nabla\hat{Q}\|_{L^{2}}^{2}
+\|\partial_{p}\hat{Q}\|_{w}^{2}.
\end{align*}
 Thus combining the above inequalities, we can deduce that
 \begin{align}\label{hat{Q}}
\frac{d}{dt}\|\hat{Q}\|_{L^{2}}^{2}+\|\nabla\hat{Q}\|_{L^{2}}^{2}
+\|\partial_{p}\hat{Q}\|_{w}^{2}
\leq C(\|\hat{v}\|_{L^{2}}^{2}+\|\hat{q}_{v}\|_{L^{2}}^{2}+\|\hat{q}_{r}\|_{L^{2}}^{2}+\|\hat{Q}\|_{L^{2}}^{2})
+\frac{1}{8\delta^{2}}\|\nabla\hat{v}\|_{L^{2}}^{2}.
\end{align}

Next, we consider the $L^{2}$ estimate for $\hat{q}_{v}$.
It is easy to check that
\begin{align}\label{hat{q}}
&\partial_{t}\hat{q}_{v}+v_{1}\cdot\nabla\hat{q}_{v}+w_{1}\partial_{p}\hat{q}_{v}+A_{q}\hat{q}_{v}
+\hat{v}\cdot\nabla q_{v2}+\hat{w}\partial_{p} q_{v2}+A_{q}\hat{q}_{v}\nonumber\\
=&f_{q_{v1}}-f_{q_{v2}}-w_{1}^{-}F^{+}(T_{1})h(q_{v1}-q_{vs})+w_{2}^{-}F^{+}(T_{2})h(q_{v2}-q_{vs}).
\end{align}
Taking the inner product of equation (\ref{hat{q}}) with $\hat{q}_{v}$ in $L^{2}$ space, we have
\begin{align*}
&\frac{1}{2}\frac{d}{dt}\|\hat{q}_{v}\|_{L^{2}}^{2}+\int_{\mathcal{M}}\hat{q}_{v}A_{q}\hat{q}_{v}d\mathcal{M}
+\int_{\mathcal{M}}(\hat{v}\cdot\nabla q_{v2}+\hat{w}\partial_{p} q_{v2})\hat{q}_{v}d\mathcal{M}\nonumber\\
=&\int_{\mathcal{M}}\left(f_{q_{v1}}-f_{q_{v2}}\right)\hat{q}_{v}d\mathcal{M}
+\int_{\mathcal{M}}\left[-w_{1}^{-}F^{+}(T_{1})h(q_{v1}-q_{vs})+w_{2}^{-}F^{+}(T_{2})h(q_{v2}-q_{vs})\right]\hat{q}_{v}d\mathcal{M}.
\end{align*}
Through integration by parts, we have
\begin{align*}
\int_{\mathcal{M}}\hat{q}_{v}A_{q}\hat{q}_{v}d\mathcal{M}\geq \|\nabla\hat{q}_{v}\|_{L^{2}}^{2}+\|\partial_{p}\hat{q}_{v}\|_{w}^{2}.
\end{align*}
Through a similar calculation as in (\ref{multiterm}), we can infer that
 \begin{align*}
\int_{\mathcal{M}}(\hat{v}\cdot\nabla q_{v2}+\hat{w}\partial_{p} q_{v2})\hat{q}_{v}d\mathcal{M}=&-\int_{\mathcal{M}}(\hat{v}\cdot\nabla \hat{q}_{v}+\hat{w}\partial_{p}\hat{q}_{v})q_{v2}d\mathcal{M}\nonumber\\
\leq&\|q_{v2}\|_{L^{\infty}}\|\hat{v}\|_{L^{2}}\|\nabla\hat{q}_{v}\|_{L^{2}}+
\|q_{v2}\|_{L^{\infty}}\|\hat{w}\|_{L^{2}}\|\partial_{p}\hat{q}_{v}\|_{L^{2}}\nonumber\\
\leq& C\|\hat{v}\|_{L^{2}}^{2}+\frac{1}{8}\|\nabla\hat{q}_{v}\|_{L^{2}}^{2}+\frac{1}{8}\|\partial_{p}\hat{q}_{v}\|_{L^{2}}^{2}
+\frac{1}{8\delta^{2}}\|\nabla\hat{v}\|_{L^{2}}^{2},
\end{align*}
where we have used the uniform boundness of $q_{v}$.
Recalling the definition of $f_{q_{v}}$, we have
\begin{align*}
&\int_{\mathcal{M}}\left(f_{q_{v1}}-f_{q_{v2}}\right)\hat{q}_{v}d\mathcal{M}
=\int_{\mathcal{M}}\left[k_{3}q_{r1}(q_{vs}-q_{v1})^{+}-k_{3}q_{r2}(q_{vs}-q_{v2})^{+}\right]\hat{q}_{v}d\mathcal{M}\nonumber\\
=&k_{3}\int_{\mathcal{M}}(q_{r1}-q_{r2})(q_{vs}-q_{v1})^{+}\hat{q}_{v}d\mathcal{M}
+k_{3}\int_{\mathcal{M}}q_{r2}\left[(q_{vs}-q_{v1})^{+}-(q_{vs}-q_{v2})^{+}\right]\hat{q}_{v}d\mathcal{M}\nonumber\\
\leq& C\|\hat{q}_{r}\|_{L^{2}}\|\hat{q}_{v}\|_{L^{2}}+C\|\hat{q}_{v}\|_{L^{2}}\|\hat{q}_{v}\|_{L^{2}}\leq C(\|\hat{q}_{r}\|_{L^{2}}^{2}+\|\hat{q}_{v}\|_{L^{2}}^{2}).
\end{align*}
Through a direct calculation, we have
\begin{align*}
&\int_{\mathcal{M}}\left[-w_{1}^{-}F^{+}(T_{1})h(q_{v1}-q_{vs})+w_{2}^{-}F^{+}(T_{2})h(q_{v2}-q_{vs})\right]\hat{q}_{v}d\mathcal{M}\nonumber\\
=&\int_{\mathcal{M}}\left[-w_{1}^{-}F^{+}(T_{1})h(q_{v1}-q_{vs})+w_{2}^{-}F^{+}(T_{1})h(q_{v1}-q_{vs})\right]\hat{q}_{v}d\mathcal{M}\nonumber\\
&+\int_{\mathcal{M}}\left[w_{2}^{-}F^{+}(T_{2})h(q_{v1}-q_{vs})-w_{2}^{-}F^{+}(T_{1})h(q_{v1}-q_{vs})\right]\hat{q}_{v}d\mathcal{M}\nonumber\\
&+\int_{\mathcal{M}}\left[w_{2}^{-}F^{+}(T_{2})h(q_{v2}-q_{vs})-w_{2}^{-}F^{+}(T_{2})h(q_{v1}-q_{vs})\right]\hat{q}_{v}d\mathcal{M}\nonumber
\end{align*}
Utilizing the uniform boundness of $F$ and $h(q_{v}-q_{vs})$, we have
\begin{align*}
&\int_{\mathcal{M}}\left[-w_{1}^{-}F^{+}(T_{1})h(q_{v1}-q_{vs})+w_{2}^{-}F^{+}(T_{1})h(q_{v1}-q_{vs})\right]\hat{q}_{v}d\mathcal{M}\nonumber\\
\leq &C\|w_{1}^{-}-w_{2}^{-}\|_{L^{2}}\|\hat{q}_{v}\|_{L^{2}}\leq C\|\hat{q}_{v}\|_{L^{2}}^{2}+\frac{1}{8}\|\nabla\hat{v}\|_{L^{2}}^{2}.
\end{align*}
Considering the Lipschitz continuity of $F$ and the uniform boundness of $h(q_{v}-q_{vs})$, we can infer that
\begin{align*}
&\int_{\mathcal{M}}\left[w_{2}^{-}F^{+}(T_{2})h(q_{v1}-q_{vs})-w_{2}^{-}F^{+}(T_{1})h(q_{v1}-q_{vs})\right]\hat{q}_{v}d\mathcal{M}\nonumber\\
\leq& C\|w_{2}^{-}\|_{L^{2}}\|F^{+}(T_{2})-F^{+}(T_{1})\|_{L^{6}}\|\hat{q}_{v}\|_{L^{3}}
\leq C\|v_{2}\|_{H^{1}}\|\hat{T}\|_{H^{1}}\|\hat{q}_{v}\|_{L^{2}}^{\frac{1}{2}}
\|\hat{q}_{v}\|_{H^{1}}^{\frac{1}{2}}\nonumber\\
\leq& C\|v_{2}\|_{H^{1}}^{4}\|\hat{q}_{v}\|_{L^{2}}^{2}+\frac{1}{8}\|\hat{q}_{v}\|_{H^{1}}^{2}
+\frac{1}{8}\|\hat{T}\|_{H^{1}}^{2}.
\end{align*}
Considering the monotonicity of the heaviside function $q_{v}\rightarrow h(q_{v}-q_{vs})$ and the positivity of $F^{+}$, we get
\begin{align*}
\int_{\mathcal{M}}\left[w_{2}^{-}F^{+}(T_{2})h(q_{v2}-q_{vs})-w_{2}^{-}F^{+}(T_{2})h(q_{v1}-q_{vs})\right]\hat{q}_{v}d\mathcal{M}\leq 0.
\end{align*}
Therefore
\begin{align*}
&\int_{\mathcal{M}}\left[-w_{1}^{-}F^{+}(T_{1})h(q_{v1}-q_{vs})+w_{2}^{-}F^{+}(T_{2})h(q_{v2}-q_{vs})\right]\hat{q}_{v}d\mathcal{M}\nonumber\\
\leq & C(1+\|v_{2}\|_{H^{1}}^{4})\|\hat{q}_{v}\|_{L^{2}}^{2}+\frac{1}{8}\|\nabla\hat{v}\|_{L^{2}}^{2}
+\frac{1}{8}\|\hat{q}_{v}\|_{H^{1}}^{2}
+\frac{1}{8}\|\hat{T}\|_{H^{1}}^{2}\nonumber\\
\leq&C(1+\|v_{2}\|_{H^{1}}^{4})\|\hat{q}_{v}\|_{L^{2}}^{2}+\frac{1}{8\delta^{2}}\|\nabla\hat{v}\|_{L^{2}}^{2}
+\frac{1}{4}\|\hat{q}_{v}\|_{H^{1}}^{2}
+\frac{1}{8}\|\hat{H}\|_{H^{1}}^{2}.
\end{align*}
Thus
\begin{align}\label{hat{q}_{v}}
&\frac{d}{dt}\|\hat{q}_{v}\|_{L^{2}}^{2}+\|\nabla\hat{q}_{v}\|_{L^{2}}^{2}+\|\partial_{p}\hat{q}_{v}\|_{L^{2}}^{2}\nonumber\\
\leq& C(1+\|v_{2}\|_{H^{1}}^{4})(\|\hat{v}\|_{L^{2}}^{2}+\|\hat{q}_{v}\|_{L^{2}}^{2}+\|\hat{q}_{r}\|_{L^{2}}^{2}+\|\hat{v}\|_{L^{2}}^{2})
+\frac{1}{4\delta^{2}}\|\nabla\hat{v}\|_{L^{2}}^{2}+\frac{1}{8}\|\hat{H}\|_{H^{1}}^{2}.
\end{align}

Next we consider the estimate for $\hat{q}_{r}$, Through a similar argument to $\hat{q}_{v}$, we can obtain that
\begin{align}\label{hat{q}_{r0}}
&\frac{1}{2}\frac{d}{dt}\|\hat{q}_{r}\|_{L^{2}}^{2}+\|\nabla\hat{q}_{r}\|_{L^{2}}^{2}+\|\partial_{p}\hat{q}_{r}\|_{L^{2}}^{2}\nonumber\\
\leq& \int_{\mathcal{M}}\left[\hat{v}\cdot\nabla q_{r2}+\hat{w}\partial_{p} q_{r2}\right]\hat{q}_{r}d\mathcal{M}+\int_{\mathcal{M}}\left(f_{q_{r1}}-f_{q_{r2}}\right)\hat{q}_{r}d\mathcal{M}.
\end{align}
Through a similar argument as in (\ref{multiterm}), we have
\begin{align}\label{hat{q}_{r1}}
\int_{\mathcal{M}}\left[\hat{v}\cdot\nabla q_{r2}+\hat{w}\partial_{p} q_{r2}\right]\hat{q}_{r}d\mathcal{M}=&-\int_{\mathcal{M}}\left[\hat{v}\cdot\nabla \hat{q}_{r}+\hat{w}\partial_{p}\hat{q}_{r}\right]q_{r2}d\mathcal{M}\nonumber\\
\leq&\|q_{r2}\|_{L^{\infty}}\|\hat{v}\|_{L^{2}}\|\nabla\hat{q}_{r}\|_{L^{2}}+
\|q_{r2}\|_{L^{\infty}}\|\hat{w}\|_{L^{2}}\|\partial_{p}\hat{q}_{r}\|_{L^{2}}\nonumber\\
\leq& C\|\hat{v}\|_{L^{2}}^{2}+\frac{1}{2}\|\nabla\hat{q}_{r}\|_{L^{2}}^{2}+\frac{1}{4}\|\partial_{p}\hat{q}_{r}\|_{L^{2}}^{2}
+\frac{1}{8\delta^{2}}\|\nabla\hat{v}\|_{L^{2}}^{2}.
\end{align}
Recalling the definition of $f_{q_{r}},$ we infer that
\begin{align}\label{hat{q}_{r2}}
&\int_{\mathcal{M}}\left(f_{q_{r1}}-f_{q_{r2}}\right)\hat{q}_{r}d\mathcal{M}\nonumber\\
\leq& C\int_{\mathcal{M}}\partial_{p}(\frac{p}{R\bar{\theta}}\hat{q}_{r})\hat{q}_{r}d\mathcal{M}
+C\int_{\mathcal{M}}\left((q_{c1}-q_{crit})^{+}-(q_{c2}-q_{crit})^{+}\right)\hat{q}_{r}d\mathcal{M}\nonumber\\
&+C\int_{\mathcal{M}}\left(q_{c1}q_{r1}-q_{c2}q_{r2}\right)\hat{q}_{r}d\mathcal{M}
+C\int_{\mathcal{M}}\left[q_{r1}(q_{vs}-q_{v1})^{+}-q_{r2}(q_{vs}-q_{v2})^{+}\right]d\mathcal{M}\nonumber\\
\leq& C(\|\hat{q}_{v}\|_{L^{2}}^{2}+\|\hat{q}_{r}\|_{L^{2}}^{2}+\|\hat{Q}\|_{L^{2}}^{2})
+\frac{1}{4}\|\partial_{p}\hat{q}_{r}\|_{L^{2}}^{2}.
\end{align}
Substituting (\ref{hat{q}_{r1}}), (\ref{hat{q}_{r2}}) into (\ref{hat{q}_{r0}}), we can deduce that
\begin{align}\label{hat{q}_{r}}
\frac{d}{dt}\|\hat{q}_{r}\|_{L^{2}}^{2}+\|\nabla\hat{q}_{r}\|_{L^{2}}^{2}+\|\partial_{p}\hat{q}_{r}\|_{L^{2}}^{2}
\leq C\left(\|\hat{v}\|_{L^{2}}^{2}+\|\hat{q}_{v}\|_{L^{2}}^{2}+\|\hat{q}_{r}\|_{L^{2}}^{2}+\|\hat{Q}\|_{L^{2}}^{2} \right)+\frac{1}{8\delta^{2}}\|\nabla\hat{v}\|_{L^{2}}^{2}.
\end{align}

Next, we consider the $L^{2}$ estimate for $H$.
 Noting that $H=T+\frac{L}{c_{p}}q_{v}$, then $H$ satisfies
\begin{align*}
\partial_{t}H+v\nabla H+w\partial_{p}H+\mathcal{A}_{q}H-\frac{RT}{c_{p}p}w=f_{T}+\frac{L}{c_{p}}f_{q_{v}}.
\end{align*}
Thus
\begin{align*}
&\partial_{t}\hat{H}+v_{1}\nabla\hat{H}+w_{1}\partial_{p}\hat{H}+\hat{v}\cdot\nabla H_{2}+\hat{w}\partial_{p}H_{2}+\mathcal{A}_{q}\hat{H}\nonumber\\
=&\frac{R}{pc_{p}}\left(T_{1}\hat{w}+w_{2}\hat{T}\right)+f_{T_{1}}+\frac{L}{c_{p}}f_{q_{v1}}-
f_{T_{2}}-\frac{L}{c_{p}}f_{q_{v2}},
\end{align*}
with the boundary condition
\begin{align*}
{\rm on}\ \Gamma_{i}:\ \partial_{p}\hat{H}=-\hat{H},\
 {\rm on}\ \Gamma_{u}:\ \partial_{p}\hat{Q}=0,\
 {\rm on}\ \Gamma_{l}:\ \partial_{\textbf{n}}\hat{Q}=-\hat{H}.
\end{align*}
Taking the inner product of the equation for $\hat{H}$ with $\hat{H}$ in $L^{2}$ space, we can deduce that
\begin{align*}
&\frac{1}{2}\frac{d}{dt}\|\hat{H}\|_{L^{2}}^{2}+\|\nabla\hat{H}\|_{L^{2}}^{2}
+\|\partial_{p}\hat{H}\|_{L^{2}}^{2}\nonumber\\
\leq& \int_{\mathcal{M}}\frac{R}{pc_{p}}\left(T_{1}\hat{w}+w_{2}\hat{T}\right)\hat{H}d\mathcal{M}+ \int_{\mathcal{M}}\left(\hat{v}\cdot\nabla H_{2}+\hat{w}\partial_{p} H_{2}\right)\hat{H}d\mathcal{M}\nonumber\\
&+\int_{\mathcal{M}}\left(f_{T_{1}}-f_{T_{2}}\right)\hat{H}d\mathcal{M}
+\frac{L}{c_{p}}\int_{\mathcal{M}}\left(f_{q_{v1}}-f_{q_{v2}}\right)\hat{H}d\mathcal{M}.
\end{align*}
Through a similar argument as in (\ref{multiterm}), we have
\begin{align*}
&\int_{\mathcal{M}}\left(\hat{v}\cdot\nabla H_{2}+\hat{w}\partial_{p} H_{2}\right)\hat{H}d\mathcal{M}=-\int_{\mathcal{M}}(\hat{v}\cdot\nabla \hat{H}+\hat{w}\partial_{p}\hat{H})H_{2}d\mathcal{M}\nonumber\\
\leq&\|H_{2}\|_{L^{\infty}}\|\hat{v}\|_{L^{2}}\|\nabla\hat{H}\|_{L^{2}}+
\|H_{2}\|_{L^{\infty}}\|\hat{w}\|_{L^{2}}\|\partial_{p}\hat{H}\|_{L^{2}}\nonumber\\
\leq& C\|\hat{v}\|_{L^{2}}^{2}+\frac{1}{2}\|\nabla\hat{H}\|_{L^{2}}^{2}+\frac{1}{4}\|\partial_{p}\hat{H}\|_{L^{2}}^{2}
+\frac{1}{8\delta^{2}}\|\nabla\hat{v}\|_{L^{2}}^{2},
\end{align*}
where we have used the uniform boundness of $H_{2}$ which can be ensured by the uniform boundness of $T$ and $q_{v}$. Recalling definitions of $f_{T},f_{q_{v}}$, we have
\begin{align*}
\int_{\mathcal{M}}\left(f_{T_{1}}-f_{T_{2}}\right)\hat{H}d\mathcal{M}
\leq C\left(\|\hat{q}_{r}\|_{L^{2}}^{2}+\|\hat{q}_{v}\|_{L^{2}}^{2}+
\|\hat{H}\|_{L^{2}}^{2}\right)+\frac{1}{8\delta^{2}}\|\nabla\hat{v}\|_{L^{2}}^{2}
\end{align*}
and
\begin{align*}
\frac{L}{c_{p}}\int_{\mathcal{M}}\left(f_{q_{v1}}-f_{q_{v2}}\right)\hat{H}d\mathcal{M}
\leq C(\|\hat{q}_{r}\|_{L^{2}}^{2}+\|\hat{q}_{v}\|_{L^{2}}^{2}+\|\hat{H}\|_{L^{2}}^{2}).
\end{align*}
Using the H\"older inequality and the Young inequality, we can infer that
\begin{align*}
&\int_{\mathcal{M}}\frac{R}{pc_{p}}\left(T_{1}\hat{w}+w_{2}\hat{T}\right)\hat{H}d\mathcal{M}\nonumber\\
\leq& C\|T_{1}\|_{L^{\infty}}\|\hat{w}\|_{L^{2}}\|\hat{H}\|_{L^{2}}
+C\|w_{2}\|_{L^{2}}\|\hat{T}\|_{L^{4}}\|\hat{H}\|_{L^{4}}\nonumber\\
\leq&C\|\hat{H}\|_{L^{2}}\|\nabla\hat{v}\|_{L^{2}}
+C\|v_{2}\|_{H^{1}}(\|\hat{H}\|_{L^{4}}+\|\hat{q}_{v}\|_{L^{4}})\|\hat{H}\|_{L^{4}}\nonumber\\
\leq&C\|\hat{H}\|_{L^{2}}\|\nabla\hat{v}\|_{L^{2}}
+C\|v_{2}\|_{H^{1}}\|\hat{H}\|_{L^{2}}^{\frac{1}{2}}\|\hat{H}\|_{H^{1}}^{\frac{3}{2}}
+C\|v_{2}\|_{H^{1}}\|\hat{q}_{v}\|_{L^{2}}^{\frac{1}{4}}\|\hat{q}_{v}\|_{H^{1}}^{\frac{3}{4}}
\|\hat{H}\|_{L^{2}}^{\frac{1}{4}}\|\hat{H}\|_{H^{1}}^{\frac{3}{4}}\nonumber\\
\leq &C\|\hat{H}\|_{L^{2}}^{2}
+C\|v_{2}\|_{H^{1}}^{4}(\|\hat{q}_{v}\|_{L^{2}}^{2}+\|\hat{H}\|_{L^{2}}^{2})
+\frac{1}{8}\|\hat{H}\|_{H^{1}}^{2}+\frac{1}{8}\|\hat{q}_{v}\|_{H^{1}}^{2}
+\frac{1}{8\delta^{2}}\|\nabla\hat{v}\|_{L^{2}}^{2},
\end{align*}
where we have used the uniform boundness of $T$ in the second step and the Gagliardo-Nirenberg-Sobolev inequality in the third step. Combining the above inequalities, we can deduce that
\begin{align}\label{hat{H}}
&\frac{d}{dt}\|\hat{H}\|_{L^{2}}^{2}+\|\nabla\hat{H}\|_{L^{2}}^{2}
+\|\partial_{p}\hat{H}\|_{L^{2}}^{2}\nonumber\\
\leq&C(1+\|v_{2}\|_{H^{1}}^{4})\left(\|\hat{v}\|_{L^{2}}^{2}+|\|\hat{q}_{r}\|_{L^{2}}^{2}+\|\hat{q}_{v}\|_{L^{2}}^{2}+
\|\hat{H}\|_{L^{2}}^{2}\right)+\frac{1}{4\delta^{2}}\|\nabla\hat{v}\|_{L^{2}}^{2}
+\frac{1}{8}\|\hat{q}_{v}\|_{H^{1}}^{2}.
\end{align}
Considering inequalities (\ref{hat{v}}), (\ref{hat{Q}}), (\ref{hat{q}_{v}}), (\ref{hat{q}_{r}}) and (\ref{hat{H}}), and setting
\begin{align*}
\Psi(t)=\delta^{2}\left(|\|\hat{q}_{r}\|_{L^{2}}^{2}+\|\hat{q}_{v}\|_{L^{2}}^{2}+
\|\hat{H}\|_{L^{2}}^{2}+\|\hat{Q}\|_{L^{2}}^{2}\right)+\|\hat{v}\|_{L^{2}}^{2},
\end{align*}
 we can deduce that
 \begin{align*}
 \frac{d}{dt}\Psi(t)\leq A_{7}(t)\Psi(t),
 \end{align*}
 where
 \begin{align*}
 A_{7}(t)=&C\left(1+\|v_{2}\|_{H^{1}}^{4}+\|\nabla v_{2}\|_{L^{2}}^{2}+\|\nabla\partial_{p}v_{2}\|_{L^{2}}^{2}+\|\partial_{p}v_{2}\|_{L^{2}}^{2}
+\|\partial_{p}v_{2}\|_{L^{2}}^{4}\right.\nonumber\\ &\left.+\|\nabla\partial_{p}v_{2}\|_{L^{2}}^{2}
+\|\partial_{p}v_{2}\|_{L^{2}}^{2}\|\nabla\partial_{p}v_{2}\|_{L^{2}}^{2}\right).
 \end{align*}
Considering the regularity of $v$, using the Gronwall inequality, we can obtain the uniqueness of $v,Q,H,q_{v},q_{r}$.
\end{proof}

\textbf{Proofs of main results}

Combining the existence result in Proposition \ref{Globalquasi-strong} and the uniqueness result in Proposition \ref{uniqueness}, we can complete the proof of Theorem 2.1. Similarly, Theorem 2.2 can also be proved.

\subsection*{Acknowledgments}

This work was supported by the Natural Science Foundation  of China (No. 12271261), the Key Research and Development Program of Jiangsu Province (Social Development) (No.  BE2019725), the Qing Lan Project of Jiangsu Province and
Postgraduate Research and Practice Innovation Program of Jiangsu Province (No. KYCX21\_0930).

\end{document}